\documentclass[a4paper,12pt]{article}

\usepackage[top=3.0cm,bottom=3.0cm,left=2.25cm,right=2.25cm]{geometry}%{geometry} 表示可以自定义页面设置
\usepackage{amsmath,amsthm,mathrsfs,graphicx,amsfonts}%{mathrsfs}用于产生一种数学用的花体字，数学公式宏包
\usepackage{bm}%处理数学公式中的黑斜体，公式中的粗体字符（用命令\boldsymbol）
\usepackage{cases}%\begin{numcases}{|x|=}x,&for$x\geq0$\\-x,&for$x<0$\end{numcases}
\usepackage[english]{babel}%算法
\usepackage{amsmath,amsthm}
\usepackage{amsfonts}
\usepackage{latexsym}
\usepackage{graphicx}%图形宏包
\usepackage{txfonts}
\usepackage[numbers,sort&compress]{natbib}
\usepackage[natural]{xcolor}
\usepackage{rotating}
\usepackage{mathtools}
\usepackage{enumitem}

\newtheorem{lem}{Lemma}[section]
\newtheorem{thm}[lem]{Theorem}

\newtheorem{prob}[lem]{Problem}
\newtheorem{claim}[lem]{\indent Claim}
\newtheorem{conj}[lem]{Conjecture}

\newtheorem{Observation}{Observation}
\begin{document}
\title{Paths of length five with equal-degree endpoints}
\date{}
%%%%%%%%%%%%%%%%%%%%%%%%%%%%%%%%%%%%%%%%%%%%%%%%%%%%
\author{Zhen Liu\footnote{Email: 1552580575@qq.com}, ~Qinghou Zeng\footnote{Research supported by National Key R\&D Program of China (Grant No. 2023YFA1010202) and National Natural Science Foundation of China (Grant No. 12371342). Email: zengqh@fzu.edu.cn (Corresponding
author)}\\
{\small Center for Discrete Mathematics, Fuzhou University, Fujian, 350003, China}}
%%%%%%%%%%%%%%%%%%%%%%%%%%%%%%%%%%%%%%%%%%%%%%%%%%%%%

%%%%%%%%%%%%%%%%%%%%%%%%%%%%%%%%%%%%%%%%%%%%%%%%%%%
\maketitle

\maketitle
\begin{abstract}
Addressing a question posed by Erd\H{o}s and Hajnal, Chen and Ma proved that, for all $n \ge 600$, the complete bipartite graph $K_{n,n+1}$ is the unique graph on $2n+1$ vertices with at least $n^2+n$ edges that contains no two vertices of equal degree joined by a path of length three. In this paper, we extend this result and show that, for all $n \ge 11$, $K_{n,n+1}$ is the unique $(2n+1)$-vertex graph with at least $n^2+n$ edges that avoids two equal-degree vertices joined by a path of length five. This confirms the very next case of a general conjecture of Chen and Ma on paths of odd length with equal-degree endpoints. 
\end{abstract}

\section{Introduction}\label{Intro}
A fundamental result in graph theory asserts that every graph with at least two vertices contains two vertices of the same degree. It is natural to ask whether these vertices can be required to satisfy additional structural constraints under suitable edge-density conditions. In 1991, Erd\H{o}s and Hajnal \cite{Erd1991} posed the following specific question (see also Problem \#816 in Thomas Bloom's collection of Erd\H{o}s problems \cite{Erd816}): 

\begin{prob} [Erd\H{o}s \cite{Erd1991}]\label{EH}
Is it true that every \((2n + 1)\)-vertex graph with \(n^{2}+n + 1\)
edges contains two vertices of the same degree that are joined by a path of length three?
\end{prob}

The complete bipartite graph $K_{n,n+1}$ shows that the edge bound in Problem~\ref{EH} is sharp, as it contains no such pairs of vertices. Recently, Chen and Ma \cite{CM2025} provided a stronger result as follows.

\begin{thm}[Chen and Ma \cite{CM2025}]\label{CM26}
Let $n\geq 600$. The unique \((2n + 1)\)-vertex graph with at least \(n^{2}+n\) edges,
that does not contain two vertices of the same degree joined by a path of length three, is the
complete bipartite graph \(K_{n,n + 1}\).
\end{thm}

%This problem is also listed as Problem \# 816 in Thomas Bloom's collection of Erd\H {o}s of problems \cite{Erd816}.
Subsequently, Liu and Zeng \cite{LZ2025} extended this result to all $n \ge 2$, thereby completely resolving the Erd\H{o}s–Hajnal problem.

Motivated by these developments, Chen and Ma \cite{CM2025} initiated the study of paths of arbitrary length. For positive integers $\ell$ and $n$, they introduced the function $p_\ell(n)$, defined as the maximum number of edges in an $n$-vertex graph containing no two vertices of equal degree joined by a path of length $\ell$. Based on this definition, they 
%determined \(p_{\ell}(n)\) for each $\ell\in\{1,2,3\}$, and 
proposed the following conjecture.

%\begin{prob}[Chen and Ma \cite{CM2025}]\label{Prob-CM}
%Determine the exact value of \(p_\ell(2n)\) for all even \(\ell\) %and sufficiently large \(n\).
%\end{prob}
\begin{conj}[Chen and Ma \cite{CM2025}]\label{Conj-Odd}
For any odd integer \(\ell \geq 3\) and sufficiently large \(n\),
\[p_\ell(2n+1) = n^2 + n.\]
\end{conj}
Theorem \ref{CM26} implies Conjecture \ref{Conj-Odd} for $\ell = 3$. In this paper, we verify Conjecture \ref{Conj-Odd} for $\ell = 5$ and establish the following theorem.

\begin{thm}\label{LZ}
Let \(n \ge 11\) be an integer. The unique \((2n+1)\)-vertex graph with at least \(n^2+n\) edges that contains no two vertices of the same degree joined by a path of length five is the complete bipartite graph \(K_{n,n+1}\). Consequently, $p_5(2n+1) = n^2 + n$ for \(n \ge 11\).
\end{thm}

%Our method can also be extended to the case of even numbers of vertices, yielding the following theorem.
%\begin{thm}\label{even}
%Let $n\ge 19$. The unique \((2n)\)-vertex graph with at least \(n^{2}-1\) edges
%that contains no two vertices of the same degree joined by a path of length five is the
%complete bipartite graph \(K_{n-1,n + 1}\).
%\end{thm}
%Since the proof is essentially identical to that of Theorem~\ref{LZ}, we leave it to the appendix to avoid redundancy.

\noindent\textbf{Notation}. Throughout the paper, we use standard graph theory notation. For any graph \(G\), \(V(G)\) denotes the vertex set of \(G\), and \(E(G)\) denotes the edge set of \(G\). For any \(v\in V(G)\), we denote by \(N_{G}(v)\) the set of \emph{neighbors} of \(v\) in \(G\), and by \(d_G(v)\) the \emph{degree} of \(v\) in \(G\). For any \(S\subseteq V(G)\), let \(G[S]\) denote the induced subgraph of \(G\) on \(S\), and let \(e(G[S])\) denote the number of edges in \(E(G[S])\). We will drop the reference to \(G\) when there is no danger of confusion.

\section{Technical lemmas}
In this section, we present some useful observations and lemmas,  all of which play crucial roles in the proof of our main theorem. 

Let $G$ be a graph. Suppose that $u$ and $v$ are two vertices of $G$, and let $\mathbf{1}_{uv}=1$ if $uv\in E(G)$ and $\mathbf{1}_{uv}=0$ otherwise. Set $B_{uv} := N(u) \cap N(v)$,
$A_u := N(u)\setminus(\{v\} \cup B_{uv})$, $ A_v := N(v)\setminus(\{u\} \cup B_{uv})$ and $D_{uv} :=V(G)\setminus(N(u)\cup N(v)\cup\{u,v\})$. We first state the following simple but frequently used observation.
\begin{Observation}\label{1}
Let $G$ be a graph. For two distinct vertices $u$ and $v$ in $G$, we have
$$|B_{uv}| = d(u) + d(v) + |D_{uv}| - |V(G)| +2\left(1-\mathbf{1}_{uv}\right).$$
\end{Observation}
\begin{proof}
Since $V(G) = \{u, v\} \cup A_u \cup A_v \cup B_{uv} \cup D_{uv}$, it follows that
\begin{align*}
|B_{uv}|
&= |V(G)| - 2 - |A_u| - |A_v| - |D_{uv}| \\
&= |V(G)| - 2 - \bigl(d(u) + d(v) - 2|B_{uv}| - 2\cdot\mathbf{1}_{uv}\bigr) - |D_{uv}| \\
&= 2|B_{uv}|-d(u) - d(v) -|D_{uv}|+|V(G)| -  2(1 - \mathbf{1}_{uv}).
\end{align*}
This implies the desired result.
\end{proof}

The next lemma serves as a key ingredient in the two core parts (see Lemmas \ref{beta 1/2} and \ref{beta n+2} in Section \ref{Main-Sec}) of the proof of Theorem \ref{LZ}.
\begin{lem}\label{dudv}
Let $G$ be a graph. Suppose that \(G\) contains no two vertices of equal degree joined by a path of length five. Let \(u\) and \(v\) be two distinct vertices of \(G\) with \(d(u) = d(v)\). Then for any vertex \(w \in N(u) \cap N(v)\), we have 
$$d(w) \le |V(G)| - d(u) + 2+\mathbf{1}_{uv}.$$
\end{lem}
\begin{proof}[{\bf Proof}]
Suppose that there exists $w \in N(u) \cap N(v)$ such that $d(w) \ge |V(G)| - d(u) + 3+\mathbf{1}_{uv}$.
Now, we partition the sets $A_u$, $A_v$, $B_{uv}$ and $D_{uv}$ with respect to $w$ as follows:
\begin{align*}
A_1 &= N(w) \cap A_u, &\, A_2 &= A_u \setminus A_1, &\,
A_3 &= N(w) \cap A_v, &\, A_4 &= A_v \setminus A_3,\\
B_1 &= N(w) \cap B_{uv}, &\, B_2 &= B_{uv} \setminus B_1, &\,
D_1 &= N(w) \cap D_{uv}, &\, D_2 &= D_{uv} \setminus D_1.
\end{align*}
Based on these sets, we first establish several claims.
\begin{claim}\label{D23}    
$|A_1|+|B_1| - |D_2|\ge 3$ and $|A_3|+|B_1| - |D_2|\ge 3$.
\end{claim}
\begin{proof}    
Note first that $A_2 \cup B_2 \cup A_4 \cup D_2 = V(G) \setminus N(w)$. Clearly, we have
\begin{align}\label{x-4}    
|A_2| + |B_2| + |A_4| + |D_2|% = |V(G)| - d(w) 
\le |V(G)| - (|V(G)|-d(u)+3+\mathbf{1}_{uv}) = d(u) - 3-\mathbf{1}_{uv}.
\end{align}
Moreover, it is easy to see that
\begin{align}\label{x-}    
|A_1| + |A_2| + |B_1| + |B_2| = d(u) - \mathbf{1}_{uv}
\quad\text{and}\quad
|A_3| + |A_4| + |B_1| + |B_2| = d(v) - \mathbf{1}_{uv}.
\end{align}
Subtracting \eqref{x-4} from the first equality in \eqref{x-} yields that
\[
|A_1| + |B_1| - (|A_4| + |D_2|) \ge (d(u) - \mathbf{1}_{uv}) - (d(u) - 3-\mathbf{1}_{uv}) =  3.
\]
This implies that $|A_1| + |B_1| - |D_2| \ge 3$.
The second inequality follows analogously by subtracting \eqref{x-4} from the second equality in \eqref{x-} as $d(u)=d(v)$.
\end{proof}
\begin{claim}\label{eA0}
The following results hold:
\begin{itemize}
    \item[\textup{(a)}] $e(A_1 \cup B_1 \cup A_3) = 0$,
    \item[\textup{(b)}]  $e\left(A_1 \cup B_1 \cup A_3, (A_2 \cup B_2 \cup A_4 \cup D_1) \setminus \{w\}\right) = 0$, 
    \item[\textup{(c)}] For every vertex $u_1 \in D_2$, we have $|N(u_1) \cap (A_1 \cup B_1 \cup A_3)| \le 1$.
\end{itemize}
\end{claim}
\begin{proof}
   \textup{(a)} Suppose that there exist two distinct vertices $u_1,u_2\in A_1\cup B_1\cup A_3$ with $u_1u_2\in E(G)$. Without loss of generality, we may assume that $u_1\in N(u)$. By Claim~\ref{D23}, there exists some vertex $u_3\in B_1\cup A_3$ different from $u_1$ and $u_2$, implying that $u_3\in N(v)$. Then $u u_1 u_2 w u_3 v$ is a path of length five with $d(u) = d(v)$, a contradiction.

\textup{(b)} Suppose that there exist $u_1 \in A_1 \cup B_1 \cup A_3$ and $u_2 \in (A_2 \cup B_2 \cup A_4 \cup D_1) \setminus \{w\}$ with $u_1u_2 \in E(G)$.
 Without loss of generality, we may assume that $u_1\in N(u)$.
If $u_2 \in D_1$, by Claim~\ref{D23} we can choose $u_3 \in B_1 \cup A_3$ distinct from $u_1$ and $u_2$.
Then $uu_1u_2wu_3v$ is a path of length five with $d(u) = d(v)$, a contradiction.
Suppose that $u_2 \in (A_2 \cup B_2 \cup A_4) \setminus \{w\}$. Without loss of generality, we may assume that $u_2 \in N(u)$.
By Claim~\ref{D23}, we can choose $u_3 \in B_1 \cup A_3$ different from $u_1$ and $u_2$.
Then $uu_2u_1wu_3v$ is a path of length five with $d(u) = d(v)$, a contradiction.

\textup{(c)}  Suppose that $u_1 \in D_2$ has two neighbors $u_2,u_3$ in $A_1\cup B_1\cup A_3$. Without loss of generality, we  may assume that $u_2\in N(u)$. Then $u u_2 u_1 u_3 w v$ is a path of length five with $d(u) = d(v)$, a contradiction.
\end{proof}

We now split the remainder of the proof into two cases based on the value of $|B_{uv}|$.

\medskip
{\bf Case 1}. $|B_{uv}|\ge 4$.
\medskip

%By Claim~\ref{D23}, we have $|A_1\cup B_1| \ge |D_2|+3$. 
For each vertex $z\in A_1\cup B_1$, we consider the number of neighbors of $z$ in $\{v\}\cup D_2$. Note that this number is between $0$ and $|D_2|+1$. Since $|A_1\cup B_1| > |D_2|+2$ by Claim~\ref{D23}, the pigeonhole principle guarantees two distinct vertices $u_1,u_2\in A_1\cup B_1$ with
\[
|N(u_1)\cap(\{v\}\cup D_2)| = |N(u_2)\cap(\{v\}\cup D_2)|.
\]
Observe that Claim~\ref{eA0}(a) and \ref{eA0}(b) imply $N(A_1 \cup B_1 \cup A_3) \subseteq \{u, v, w\} \cup D_2$.
Moreover, $u$ and $w$ are adjacent to every vertex in $A_1 \cup B_1$.
Consequently, $d(u_1) = d(u_2)$.
Choose a vertex $u_3 \in B_{uv}$ different from $u_1$, $u_2$ and $w$ (such a vertex exists as $|B_{uv}| \ge 4$).
Then $u_1uu_3vw u_2$ is a path of length five with $d(u_1) = d(u_2)$, a contradiction.

\medskip
{\bf Case 2}. $|B_{uv}|\le 3$.
\medskip

Since $B_1 \subseteq B_{uv}\setminus\{w\}$ and $|B_{uv}| \le 3$, we have $|B_1| \le 2$.
This together with Claim~\ref{D23} implies that $|A_1| > |D_2|$ and $|A_3| > |D_2|$.
By Claim~\ref{eA0}(c), each vertex in $D_2$ is adjacent to at most one vertex in $A_1 \cup B_1 \cup A_3$. Hence $e(D_2, A_1 \cup B_1 \cup A_3) \le |D_2|$.
Combining $|A_1| > |D_2|$ and $|A_3| > |D_2|$, we can select $w_1 \in A_1$ and $w_2 \in A_3$ such that $N(w_1) \cap D_2 = N(w_2) \cap D_2 = \emptyset$.
Claim~\ref{eA0}(a) and \ref{eA0}(b) imply that $w_1$ and $w_2$ have no neighbors in $V(G) \setminus (\{u, v, w\} \cup D_2)$, and we have already excluded neighbors in $D_2$.
Thus, $N(w_1) = \{u, w\}$ and $N(w_2) = \{v, w\}$, i.e., $d(w_1) = d(w_2)$.
Since $|A_3 \cup B_1| \ge 3$ by Claim~\ref{D23}, there exists a vertex $w_3 \in (A_3 \cup B_1) \setminus \{w_2\}$.
Then $w_1uw w_3vw_2$ is a path of length five with $d(w_1) = d(w_2)$, a contradiction.
\end{proof}

The last lemma is implicit in Chen and Ma \cite{CM2025} (see the proof of Lemma 5) and can be stated as follows. We include its proof here for the sake of completeness. Throughout this paper, we set $\binom{x}{2} = 0$ for any integer $x<2$.

\begin{lem}[Chen and Ma \cite{CM2025}]\label{lambda}
Let $\nu,\Delta,\beta$ and $b$ be positive integers such that $\Delta\ge b$ and $\nu\ge \Delta+1$, and let
\[
\lambda(\nu,\Delta,\beta,b):=\max_{\mathcal{A}, \mathcal{B}} \sum_{k \in \mathcal{A} \cup \mathcal{B}} k,
\]
where
\begin{itemize}
\item[(1)] $\mathcal{A}$ is a sequence of $\Delta-b$ distinct integers in $\{1, 2, \dots, \Delta - 1\}$;
\item[(2)] $\mathcal{B}$ is a sequence of $\nu - \Delta-1$ integers in $\{0,1,\dots, \Delta - 1\}$;
\item[(3)] no two elements in $\mathcal{A} \cup \mathcal{B}$ share the same value greater than $\beta$.
\end{itemize}
Then
\[
\lambda(\nu,\Delta,\beta,b)=\sum_{k=b}^{\Delta-1} k + (\nu-1-\Delta)\beta+\binom{b-\beta}{2}-\binom{(b-\beta)-(\nu-1-\Delta)}{2}.
\]
\end{lem}
\begin{proof}[{\bf Proof}]
We distinguish three cases according to the value of $b-\beta$.

\medskip
\textbf{Case 1}. $b-\beta\le0$.
\medskip

We argue that  $\mathcal{A} \supseteq \{\beta, \beta+1,\dots,\Delta-1\}$ when the maximum is achieved.
Otherwise, there exist $k\in\mathcal{A}$ and $\ell\notin\mathcal{A}$ with $k<\beta\le\ell\le\Delta-1$.
If $\ell$ is not included in $\mathcal{B}$, then we replace $k$ with $\ell$ in $\mathcal{A}$.
If $\ell$ is included in $\mathcal{B}$, then we replace $k$ with $\ell$ in $\mathcal{A}$ and also replace $\ell$ with $\beta$ in $\mathcal{B}$.
In both cases, conditions (1), (2), and (3) remain satisfied,
but the sum $\sum_{k\in\mathcal{A}\cup\mathcal{B}}k$ strictly increases, a contradiction.
Since $\mathcal{A}\supseteq\{\beta,\beta+1,\dots,\Delta-1\}$ and $\mathcal{A}$ satisfies condition (1),
we easily obtain $\mathcal{A}=\{b,b+1,\dots,\Delta-1\}$ and all elements in $\mathcal{B}$ must be $\beta$. Then
\begin{align*}
    \lambda(\nu,\Delta,\beta,b)=\sum_{k=b}^{\Delta-1}k+(\nu-1-\Delta)\beta.
\end{align*}

\medskip
{\bf Case 2}.  $b-\beta\ge \nu-\Delta$.
\medskip

Note that the combined sequence $\mathcal{A} \cup \mathcal{B}$ has $\nu-1 -b$ elements.
Since $b-\beta\ge (\nu-1-\Delta)+ 1$, we have $\Delta + b - (\nu-1)\ge \beta+1$.
Condition (3) implies that the optimal configuration for $\mathcal{A} \cup \mathcal{B}$ is
$\{\Delta-1, \Delta-2, \dots, \Delta + b - (\nu -1)\}$,
for which conditions (1) and (2) remain satisfied. Thus
\begin{align*}
\lambda(\nu, \Delta, \beta, b) = \sum_{k = \Delta + b - (\nu-1)}^{\Delta - 1} k.
\end{align*}

\medskip
\textbf{Case 3}. $1\le b-\beta\le \nu-\Delta-1$.
\medskip

With the optimization reasoning from the previous cases in mind,
it is straightforward to see that the optimal configuration in this case consists of all elements of $\mathcal{A}$ and some elements of $\mathcal{B}$ forming $\{\Delta-1, \Delta-2, \dots, \beta+1\}$,
while the remaining elements of $\mathcal{B}$ are all $\beta$.
Therefore
\[
\lambda(\nu, \Delta, \beta, b)=\sum_{k=\beta+1}^{\Delta-1}k + \big((\nu-1)-b-\Delta+\beta+1\big)\beta.
\]

Combing the above three cases, the desired result follows. This completes the proof of Lemma \ref{lambda}.
\end{proof}

\section{Proof of Theorem \ref{LZ}}\label{Main-Sec}
Throughout this section, let \(n \ge 11\) be an integer and let \(G\) be a graph on \(2n+1\) vertices with at least \(n^2+n\) edges. Assume that \(G\) contains no two vertices of equal degree joined by a path of length five. Our goal is to show that \(G\) must be the complete bipartite graph \(K_{n,n+1}\).

%To complete our proof, we establish the following two lemmas. Both results play key roles in our proof; however, the first one represents the most technical part of our argument, and the second one employs the same method as that of Chen and Ma \cite{CM2025}.

To complete our proof, we establish the following two lemmas. Both play key roles: the first constitutes the most technical part of our argument, whereas the second follows the same method as Chen and Ma \cite{CM2025}—despite the crucial use of Lemma \ref{dudv}. Let \(\beta\) be the largest integer such that \(G\) contains two vertices of degree \(\beta\). 

\begin{lem}\label{beta 1/2}
    We have \(\beta \leq n + 1\). Moreover, if \(\beta = n + 1\), then \(G\) is isomorphic to \(K_{n,n+1}\).
\end{lem}

\begin{lem}\label{beta n+2}
We have either $\beta \geq \Delta-3$ or $\Delta\leq n+3$.
\end{lem}

We now give a proof of Theorem \ref{LZ}, postponing the proofs of Lemmas~\ref{beta 1/2} and \ref{beta n+2} to the next two subsections.
\begin{proof}[\bf{Proof of Theorem \ref{LZ}}]
Lemmas~\ref{beta 1/2} and \ref{beta n+2} imply that either (i) $\beta = n+1$ and $G$ is isomorphic to $K_{n,n+1}$, or (ii) $\beta \le n$ and $\Delta \le n+3$.
However, the second case yields that
$$2e(G)=\sum_{v\in V(G)} d(v) \le (2n+1-3)n+\sum_{k=n+1}^{n+3} k= 2n^2 + n + 6.$$
Since $n \ge 11$, this contradicts our assumption that $e(G)\ge n^2 + n$.
Therefore, we conclude that $G$ must be the complete bipartite graph $K_{n,n+1}$.
This completes the proof of Theorem~\ref{LZ}.
\end{proof}

\subsection{Proof of Lemma \ref{beta 1/2}}
In this subsection, we give a proof of Lemma \ref{beta 1/2}.
\begin{proof}[\bf{Proof of Lemma \ref{beta 1/2}}]
    Suppose  for a contradiction that \(\beta = n + c\) with \(c \geq 1\). Choose two vertices \(u, v \in V(G)\) such that \(d(u) = d(v) = n + c\).
 Let \(|B_{uv}| = x\). Then \(|A_u| =|A_v|=n+c-x-\textbf{1}_{uv}\) and \(|D_{uv}| = 2n-1-|B_{uv}| -|A_u|-|A_v|=x-1-2c+2\cdot\textbf{1}_{uv}\). Clearly, we have $|D_{uv}| < |B_{uv}|$ as $c \ge 1$.

\begin{claim}\label{D2N+1}
   Let $|D_{uv}| \neq 0$ and suppose that there exists a vertex $w \in D_{uv}$ with $d(w) \ge n+1$. Then for every vertex $w_1 \in A_u \cup A_v \cup B_{uv}$, either (i) $d(w)+d(w_1) \le 2n$, or (ii) $d(w)+d(w_1)=2n+1$ and $ww_1 \in E(G)$.
\end{claim}
\begin{proof}
   Assume, to the contrary, that there exists some vertex $w_1 \in A_u \cup A_v \cup B_{uv}$ such that $d(w)+d(w_1) \ge 2n+2$, or $d(w)+d(w_1)=2n+1$ but $ww_1 \notin E(G)$.
Without loss of generality, we may assume that $w_1 \in N(u)$.
Applying Observation \ref{1} to the pair $(w,w_1)$, it is easy to  obtain that $|N(w)\cap N(w_1)| \ge 1$ by our assumption. Thus, there exists a vertex $w_2 \in N(w)\cap N(w_1)$.
Since $w\in D_{uv}$, we have $wu,wv\notin E(G)$.
Now apply Observation \ref{1} to $(w,v)$. It follows from $d(w)\ge n+1$ and $d(v)=\beta\ge n+1$ that $|N(w)\cap N(v)| \ge 3$.
Hence, we can find a vertex $w_3\in N(w)\cap N(v)$ distinct from $w_1$ and $w_2$.
Then $u w_1 w_2 w w_3 v$ is a path of length five whose endpoints have the same degree, a contradiction.
\end{proof}
\begin{claim}\label{2.9}
  If $|D_{uv}|\neq 0$, then we have $d(w)\le n+1$  for any vertex $w\in D_{uv}$.
\end{claim}
\begin{proof}
Choose $w\in D_{uv}$ such that $d(w)$ is maximal, and assume that $d(w) \ge n+2$. By Claim \ref{D2N+1}, for every $w_1\in A_u\cup B_{uv}\cup A_v$, we have $d(w)+d(w_1)\le 2n+1$. Consequently, $d(w_1)\le n-1$ and $d(z)+d(w_1)\le 2n+1$  for any $z\in D_{uv}$ and any $w_1\in A_u\cup B_{uv}\cup A_v$. Recall that $|D_{uv}| < |B_{uv}|$.
%and consequently $d(w_1)\le n-1$ because $d(w)\ge n+2$. Moreover, for any $z\in D_{uv}$ and any $w_1\in A_u\cup B_{uv}\cup A_v$, we have $d(z)+d(w_1)\le d(w)+d(w_1)\le 2n+1$.
%Recall that $|D_{uv}| = x-1-2c+2\mathbf{1}_{uv}$ and $|A_u\cup B_{uv}\cup A_v| = 2n-1-|D_{uv}|$. Note that $|D_{uv}| < |B_{uv}| = x$.
Now we can pair each vertex of $D_{uv}$ with a distinct vertex of $A_u\cup B_{uv}\cup A_v$, leaving $|A_u\cup B_{uv}\cup A_v| - |D_{uv}|$ vertices unpaired. For each paired couple $(z,w_1)$, we have $d(z)+d(w_1)\le 2n+1$, and each unpaired vertex has degree at most $n-1$. Let $D'$ be the set of vertices in $A_u\cup B_{uv}\cup A_v$ paired with $D_{uv}$. Therefore,
\begin{align*}    
\sum_{z\in V(G)} d(z) 
&=  d(u)+d(v) +\sum_{z\in D'\cup D_{uv}} d(z) + \sum_{z\in (A_u\cup B_{uv}\cup A_v)\setminus D'} d(z) \\
&\le d(u)+d(v)+|D_{uv}|(2n+1) + \left(|A_u\cup B_{uv}\cup A_v| - |D_{uv}|\right)(n-1)  \\
&= 2(n+c) +\left(x-1-2c+2\cdot\mathbf{1}_{uv}\right)(2n+1) + \left(2n-1-2\left(x-1-2c+2\cdot\mathbf{1}_{uv}\right)\right)(n-1)  \\
&= 2n^2 - n - 4c + 3x + 6\cdot\mathbf{1}_{uv} - 2 \\
&\le 2n^2+2n-c+3\cdot\mathbf{1}_{uv}-2. \label{n+1}
\end{align*}
The last inequality holds as $x=|B_{uv}|\le d(u)- \mathbf{1}_{uv}=n + c - \mathbf{1}_{uv}$.
It follows from $ e(G) \ge n^2+n$ that $3\cdot\mathbf{1}_{uv} - c - 2\geq 0.$ Since $\mathbf{1}_{uv} \in \{0,1\}$ and $c \ge 1$,
we conclude that $c = 1$ and $\mathbf{1}_{uv} = 1$.
Substituting these equalities into the previous estimates,
we find that every intermediate inequality must hold with equality.
Hence we obtain $x = n + c - \mathbf{1}_{uv} = n$
and consequently $|D_{uv}| = x - 1 - 2c + 2\cdot\mathbf{1}_{uv} = n - 1$.

Now $|D_{uv}|=n-1\ge 2$ (since $n\ge 11$). Recall that $d(w)$ is maximal and $d(w) \ge n+2$. If every vertex in $D_{uv}$ had degree $d(w)$, then $d(w)$ would appear at least twice, contradicting the definition of $\beta$. Thus there exists a vertex $z\in D_{uv}$ distinct from $w$ such that $d(z) < d(w)$. For this $z$, we have $d(z)+d(w_1) < d(w)+d(w_1) \le 2n+1$ for every $w_1\in A_u\cup B_{uv}\cup A_v$. This means that the required equality does not hold for the pair $(z,w_1)$,
%Consequently, in the pairing, the sum for the couple containing $z$ is strictly less than $2n+1$, 
a contradiction. Thus, we complete the proof of this claim.
\end{proof}
\begin{claim}\label{2.10}
    If $|D_{uv}| \neq 0$, then either $d(w) \le n$ for every vertex $w \in D_{uv}$, or $G$ is isomorphic to $K_{n,n+1}$.
\end{claim}
\begin{proof}
 By Claim \ref{2.9}, we have $d(w) \le n+1$ for every vertex $w \in D_{uv}$. Suppose that there exists a vertex $w\in D_{uv}$ with $d(w)=n+1$. By Claim~\ref{D2N+1}, for any $w_1 \in A_u \cup B_{uv} \cup A_v$,
 we have either $d(w_1) \le n-1$,
or $d(w_1) = n$ and $ww_1 \in E(G)$ as $d(w) = n+1$.
In particular, $d(w_1) \le n$ for all $w_1 \in A_u \cup B_{uv} \cup A_v$.

\bigskip
{\bf Case 1}. There exists $w_1\in A_u\cup B_{uv}\cup A_v$ with $d(w_1)=n$ and $ww_1 \in E(G)$.
\medskip

We first show that $w_1$ has no neighbor in $A_u\cup B_{uv}\cup A_v$. Suppose, for contradiction, that there exists $w_2\in N(w_1)\cap (A_u\cup B_{uv}\cup A_v)$. Without loss of generality, assume $w_2\in N(u)$. Since $w\in D_{uv}$, we have $wu,wv\notin E(G)$. Applying  Observation \ref{1} to $(w,v)$ yields $|N(w)\cap N(v)|\ge 3$. Hence we can pick $w_3\in N(w)\cap N(v)$ distinct from $w_1$ and $w_2$. Then $u w_2 w_1 w w_3 v$ is a path of length five whose endpoints $u$ and $v$ have the same degree, a contradiction. Thus, $N(w_1)\cap (A_u\cup B_{uv}\cup A_v)=\emptyset$.

Since $d(w_1) = n$ and $N(w_1) \cap (A_u \cup B_{uv} \cup A_v) = \emptyset$, we have $|N(w_1) \cap D_{uv}| \ge n - 2$.
Consequently, $|D_{uv}| \ge n - 2$.
Recall that $|D_{uv}| = x - 1 - 2c + 2\cdot\mathbf{1}_{uv}$. This together with  $x \le n + c - \mathbf{1}_{uv}$ implies that
\begin{equation}\label{eq:Duv_bound}
|D_{uv}|  \le n - c - 1 + \mathbf{1}_{uv}.
\end{equation}

If $\mathbf{1}_{uv}=0$, then $|D_{uv}|\le n-c-1\le n-2$. This together with $|D_{uv}| \ge n - 2$ implies that $|D_{uv}|=n-2$ and $c=1$. It follows that $|B_{uv}|=x=n+1$ and $|A_u|=|A_v|=0$.
Note that every vertex in $D_{uv}\cup\{u,v\}$ has degree at most $n+1$
and every vertex in $B_{uv}$ has degree at most $n$.
%with $d(u)=d(v)=n+1$,
If there exists a vertex in $B_{uv}$ with degree at most $n-1$,
then
\[
\sum_{z\in V(G)} d(z) \le (|D_{uv}|+2)(n+1) + n|B_{uv}| - 1 = 2n^2+2n-1,
\]
a contradiction.
This implies that all $x = n+1$ vertices in $B_{uv}$ have degree  $n$.
Recall that $N(w_1)\cap B_{uv}=\emptyset$ for any vertex $w_1 \in B_{uv}$ with $d(w_1)=n$. This implies that both $B_{uv}$ and $D_{uv}$ are independent sets.
A straightforward verification shows that $G$ is isomorphic to $K_{n,n+1}$.

If $\mathbf{1}_{uv}=1$, then \eqref{eq:Duv_bound} implies that $|D_{uv}| \le n - c$.
Since $|D_{uv}| \ge n - 2$, we obtain $1 \le c \le 2$. We now argue that  no two vertices in $D_{uv}$ can have same degree $n+1$. Indeed, if $u_1,u_2\in D_{uv}$ both have degree $n+1$, then applying  Observation \ref{1} to the pairs $(u,u_1)$ and $(v,u_2)$ yields at least three common neighbors for each pair. Suppose that $w_1 \in N(u_1) \cap N(u)$ and $w_2 \in N(u_2) \cap N(v) \setminus \{w_1\}$.
Note also that $u_1$ and $u_2$ belong to $D_{uv}$.
Then $u_1 w_1 u v w_2 u_2$ is a path of length five whose endpoints $u_1$ and $u_2$ have the same degree $n+1$,
a contradiction. Therefore, at most one vertex in $D_{uv}$ has degree $n+1$, and all others have degree at most $n$. This together with $|D_{uv}| = x - 1 - 2c + 2\cdot\mathbf{1}_{uv}$ implies that
\begin{align*}
    \sum_{z\in V(G)} d(z) &\le (|D_{uv}|-1)n + (n+1) + (2n-1-|D_{uv}|)n + d(u)+d(v)\\&=(x-2c)n+(n+1)+(2n-1-(x+1-2c))n+2(n+c)\\
    &=2n^2+n+2c+1.
\end{align*}
Since  $e(G)\ge n^2 + n$, we have $n\le 2c+1$. This leads to a contradiction as $n \ge 11$ and $1 \le c \le 2$.

\bigskip
{\bf Case 2}. For every vertex $w \in A_u\cup B_{uv}\cup A_v$, we have $d(w) \le n-1$.
\medskip

Since every vertex in $D_{uv}$ has degree at most $n + 1$
and every vertex in $A_u \cup B_{uv} \cup A_v$ has degree at most $n - 1$,
we conclude that
\begin{align*}
\sum_{z\in V(G)} d(z)
&\le |D_{uv}|(n + 1) + (2n - 1 - |D_{uv}|)(n - 1) + d(u) + d(v)\\
&= (x - 1 - 2c + 2\cdot\mathbf{1}_{uv})(n + 1)
 + (2n - 1 - (x - 1 - 2c + 2\cdot\mathbf{1}_{uv}))(n - 1) + 2(n + c)\\
&= 2n^2 - n - 2c + 2x + 4\cdot\mathbf{1}_{uv} - 1\\
&\le 2n^2 + n + 2\cdot\mathbf{1}_{uv} - 1,
\end{align*}
where the last inequality holds due to $x \le n + c - \mathbf{1}_{uv}$ .
Since $n \ge 11$ and $\mathbf{1}_{uv} \in \{0, 1\}$, this contradicts our assumption that $e(G)\ge n^2 + n$.
\end{proof}

\begin{claim}\label{2.11}
    For any two vertices $w_1, w_2 \in A_u \cup B_{uv} \cup A_v$ with $d(w_1) \ge n + 1$ and $d(w_2) \ge n + 1$, we have $w_1w_2 \in E(G)$.
\end{claim}
\begin{proof}
Suppose that $w_1w_2\notin E(G)$. Without loss of generality, we may further assume that $w_1\in N(u)$. By Observation \ref{1} and $d(w_1),d(w_2)\geq n+1$, we have
\begin{align}\label{ge3}
    |N(w_1)\cap N(w_2)|\geq3.
\end{align}
So there exists $u_1\in(N(w_1)\cap N(w_2))\setminus\{u,v\}$.
Since $w_1w_2 \notin E(G)$, we have $w_1 \notin N(w_2) \cap N(v)$. If there exists $u_2 \in (N(w_2) \cap N(v)) \setminus \{u, u_1\}$, then $uw_1u_1w_2u_2v$ is a path of length five whose endpoints $u$ and $v$ have the same degree, a contradiction. Thus,
\begin{align}\label{w_2v}    
N(w_2) \cap N(v) \subseteq \{u, u_1\}.
\end{align}
If $vw_2 \notin E(G)$, then Observation \ref{1} implies $|N(w_2) \cap N(v)| \ge 3$, which contradicts \eqref{w_2v}. Therefore, $vw_2 \in E(G)$.
Similarly, since $w_1w_2 \notin E(G)$, we have $w_2 \notin N(w_1) \cap N(u)$. If there exists $u_2 \in (N(w_1) \cap N(u)) \setminus \{v, u_1\}$, then $uu_2w_1u_1w_2v$ is a path of length five whose endpoints $u$ and $v$ have the same degree, a contradiction. Thus,
\begin{align}\label{w_1u}    
N(w_1) \cap N(u) \subseteq \{v, u_1\}.
\end{align}

We next prove that $w_1v, w_2u \in E(G)$.
Without loss of generality, we may assume that $w_1v \notin E(G)$.
It follows that $v \notin N(w_1)\cap N(w_2)$ and $v \notin N(w_1)\cap N(u)$.
By \eqref{ge3}, there exist distinct vertices $u_1, u_2 \in (N(w_1)\cap N(w_2))\setminus\{u, v\}$.
By  Observation \ref{1}, there exists $u_3 \in (N(w_1) \cap N(u)) \setminus \{v\}$. This together with \eqref{w_1u} implies that $\{u_1\} = \{u_3\} = N(w_1) \cap N(u)$. Furthermore, since $w_2v \in E(G)$, $uu_3w_1u_2w_2v$ is a path of length five whose endpoints $u$ and $v$ have the same degree, a contradiction.

Now we proceed by distinguishing two cases based on whether $d(w_2) = d(v)$ or not.

{\bf Case 1}. $d(w_2) \neq d(v)$.

Since $d(w_2) \ge n + 1$, $d(v) \ge n + 1$, and $d(w_2) \neq d(v)$, we have $d(w_2) + d(v) \ge 2n + 3$. Applying Observation~\ref{1} to the pair $(w_2, v)$ yields that
\[
|N(w_2) \cap N(v)|  \ge 2. 
\]
By \eqref{w_2v}, we have $N(w_2) \cap N(v) = \{u, u_1\}$, which implies $u_1v \in E(G)$. If there exists a vertex $u_2 \in N(w_1) \cap N(w_2)$ with $u_2 \notin \{u, v, u_1\}$, then $v u_1 w_2 u_2 w_1 u$ is a path of length five whose endpoints $u$ and $v$ have the same degree, a contradiction. Hence, by \eqref{ge3}, we have $N(w_1) \cap N(w_2) = \{u, v, u_1\}$.  
This together with Observation \ref{1} implies that $D_{w_1w_2} = \emptyset$.
%, since $N(w_1) \cap N(w_2) = \{u, v, u_1\}$ and $N(w_2) \cap N(v) = \{u, u_1\}$, we get $D_{w_1w_2} = \emptyset$ and $D_{w_2v} = \emptyset$.
Since $d(v) \ge n + 1$, we obtain
\[
|N(v) \setminus (D_{w_1w_2} \cup \{w_1, w_2\})| \ge n - 1.
\]
Note also that $N(v) \setminus (D_{w_1w_2} \cup \{w_1, w_2\}) \subseteq N(w_1) \cup N(w_2)$.
This deduces that either $|N(v) \cap N(w_1)| \ge \frac{n - 1}{2}$ or $|N(v) \cap N(w_2)| \ge \frac{n - 1}{2}$, which together with \eqref{w_2v} yields that $|N(v) \cap N(w_1)| \ge 5$ as $n \ge 11$.
%Since $n \ge 17$, we conclude that either $|N(v) \cap N(w_1)| \ge 8$ or $|N(v) \cap N(w_2)| \ge 8$.
%By \eqref{w_2v}, we further get $|N(v) \cap N(w_1)| \ge 8$.
Thus, we can find a vertex $u_2 \in N(w_1) \cap N(v)$ distinct from $u, w_2, u_1$.
It follows from $uw_2 \in E(G)$ that $v u_2 w_1 u_1 w_2 u$ is a path of length five whose endpoints $u$ and $v$ have the same degree, a contradiction.

{\bf Case 2}. $d(w_2) = d(v)$.

Suppose that there exists a vertex $u_2$ distinct from $u_1, w_1, w_2$ such that $u_2 \in N(u) \cap N(v)$. Then $w_2u_1w_1uu_2v$ is a path of length five whose endpoints $w_2$ and $v$ have the same degree, a contradiction. Thus,
$
|N(u) \cap N(v)| \le 3.
$
By Observation \ref{1}, we have $|D_{uv}| \le 2$. Since $d(w_1) \ge n + 1$, it follows that
\[
|N(w_1) \setminus (D_{uv} \cup \{u, v\})| \ge n - 3.
\]
Note also that
$N(w_1) \setminus (D_{uv} \cup \{u, v\}) \subseteq N(u) \cup N(v).$
Therefore, either $|N(w_1) \cap N(u)| \ge \frac{n - 3}{2}$ or $|N(w_1) \cap N(v)| \ge \frac{n - 3}{2}$, yielding that $|N(w_1) \cap N(v)| \ge 4$ in view of \eqref{w_1u} and $n \ge 11$.
%As $n \ge 17$, we have either $|N(w_1) \cap N(u)| \ge 7$ or $|N(w_1) \cap N(v)| \ge 7$.  
%Without loss of generality, assume $|N(w_1) \cap N(v)| \ge 7$. 
Then there exists a vertex $u_2$ distinct from $u_1, u, w_2$ such that $u_2 \in N(w_1) \cap N(v)$. Thus, $uw_2u_1w_1u_2v$ is a path of length five whose endpoints $u$ and $v$ have the same degree, a contradiction.
\end{proof}
\begin{claim}\label{2.15}
    The number of vertices in $A_u \cup B_{uv} \cup A_v$ with degree at least $n + 1$ is at most two.
\end{claim}
\begin{proof}
   Suppose, for contradiction, that there exist three distinct vertices $w_1,w_2,w_3\in A_u\cup B_{uv}\cup A_v$ with $d(w_i)\ge n+1$ for $i=1,2,3$. By Claim \ref{2.11}, they are pairwise adjacent. Without loss of generality, we may assume that $w_1\in N(u)$.

We first show that  $w_2\in N(v)$. Suppose that $w_2v\notin E(G)$. Applying Observation \ref{1} to $(w_2,v)$ yields that $|N(w_2)\cap N(v)|\ge 3$. If there exists $u_1\in N(w_2)\cap N(v)$ with $u_1\notin\{w_1,w_3,u\}$, then $uw_1w_3w_2u_1v$ is a path of length five whose endpoints $u$ and $v$ have the same degree, a contradiction. Hence $N(w_2)\cap N(v)=\{w_1,w_3,u\}$, so $w_1,w_3\in N(v)$ and $uv\in E(G)$. Now consider $N(w_1)\cap N(w_2)$. If there exists $y\in N(w_1)\cap N(w_2)$ with $y\notin\{w_3,u\}$, then $uw_2yw_1w_3v$ is a path of length five whose endpoints $u$ and $v$ have the same degree, a contradiction. Thus $N(w_1)\cap N(w_2)\subseteq\{w_3,u\}$. By Observation \ref{1}, we obtain $|D_{w_1w_2}|\le 1$. Since $d(w_3) \ge n + 1$, it follows that
\[
|N(w_3) \setminus (D_{w_1w_2} \cup \{w_1, w_2\})| \ge n - 2.
\]
This together with $n \ge 11$ implies that either $|N(w_3) \cap N(w_1)| \ge \frac{n - 2}{2} >4$ or $|N(w_3) \cap N(w_2)| \ge \frac{n - 2}{2} > 4$.  
Without loss of generality, we may assume that $|N(w_3) \cap N(w_1)| > 4$. Then there exists a vertex $u_1 \in N(w_3) \cap N(w_1)$ distinct from $w_2, u, v$. Thus, $uw_2w_3u_1w_1v$ is a path of length five whose endpoints $u$ and $v$ have the same degree, a contradiction. Therefore, $w_2v \in E(G)$. By symmetry, $w_3v\in E(G)$.

Now, we have $\{w_2, w_3\} \subseteq N(v)$. If there exists a vertex $w \in N(w_1) \cap N(w_2) \setminus \{w_3, v, u\}$. Then $uw_1ww_2w_3v$ is a path of length five whose endpoints $u$ and $v$ have the same degree, a contradiction. Thus, $N(w_1) \cap N(w_2) \subseteq \{w_3, v, u\}$. By Observation~\ref{1}, we have $|D_{w_1w_2}| \le 2$. 
Since $d(w_3) \ge n + 1$, it follows that
\[
|N(w_3) \setminus (D_{w_1w_2} \cup \{w_1, w_2\})| \ge n - 3.
\]
This implies that either $|N(w_3) \cap N(w_1)| \ge \frac{n - 3}{2} \ge 4$ or $|N(w_3) \cap N(w_2)| \ge \frac{n - 3}{2} \ge 4$ as $n \ge 11$.
Without loss of generality, we may assume that $|N(w_3) \cap N(w_1)| \ge 4$. Then there exists a vertex $u_1 \in N(w_3) \cap N(w_1)$ distinct from $w_2, u, v$. Thus, $vw_2w_3u_1w_1u$ is a path of length five whose endpoints $u$ and $v$ have the same degree, a contradiction. 
This completes the proof of this claim.
\end{proof}
\begin{claim}\label{2n+4}
    For any vertices $w_1, w_2 \in A_u \cup B_{uv} \cup A_v$, we have $d(w_1) + d(w_2) \le 2n + 4$.
\end{claim}
\begin{proof}
    Suppose there exist two vertices $w_1, w_2 \in A_u \cup B_{uv} \cup A_v$ such that $d(w_1) + d(w_2) \ge 2n + 5$. Since $d(w_1)$ and $d(w_2)$ are integers, at least one of $d(w_1)$ and $d(w_2)$ is at least $n + 3$. Without loss of generality, assume $d(w_1) \ge n + 3$. 

Since $w_2 \in A_u \cup B_{uv} \cup A_v$, without loss of generality, we may assume that $w_2 \in N(u)$.
As $d(w_1) \ge n + 3$ and $d(v) \ge n + 1$, by Observation~\ref{1}, we have $|N(w_1) \cap N(v)| \ge 3$.
Thus, there exists a vertex $w$ distinct from $w_2$ and $u$ such that $w \in N(w_1) \cap N(v)$. 
Since $d(w_1) + d(w_2) \ge 2n + 5$, by Observation~\ref{1}, we have $|N(w_1) \cap N(w_2)| \ge 4$.
Therefore, there exists a vertex $u_1$ distinct from $u$, $v$, and $w$ such that $u_1 \in N(w_1) \cap N(w_2)$.
Then $uw_2u_1w_1wv$ is a path of length five whose endpoints $u$ and $v$ have the same degree, a contradiction.
\end{proof}

By Claim \ref{2n+4}, the sum of the degrees of the two vertices with maximum degree in $A_u \cup B_{uv} \cup A_v$ is at most $2n + 4$. Claim \ref{2.15} implies that there are at most two vertices in $A_u \cup B_{uv} \cup A_v$ having degree at least $n + 1$. Consequently, all other vertices in this set have degree at most $n$. Hence,
\begin{align}\label{2n+4.}
    \sum_{z\in A_u \cup B_{uv} \cup A_v} d(z) \le (2n - 1 - |D_{uv}| - 2)n + (2n + 4).
\end{align}
By Lemma~\ref{dudv}, we also have $d(w) \le n - c + 3+\textbf{1}_{uv}$ for any $w \in B_{uv}$.
If $c\ge3+\textbf{1}_{uv}$, then 
\begin{align}\label{2n+4..}
\sum_{z\in A_u \cup B_{uv} \cup A_v} d(z)
&\le |B_{uv}|(n - c + 3+\textbf{1}_{uv}) + (2n - 1 - |D_{uv}| - 2 - |B_{uv}|)n + (2n + 4)\notag\\
&=(2n - 1 - |D_{uv}| - 2)n + (2n + 4)-|B_{uv}|(c-3-\textbf{1}_{uv}).
\end{align}
By Claim \ref{2.10}, every vertex in $D_{uv}$ has degree at most $n$ if $G$ is not isomorphic to $K_{n,n+1}$. Thus
\begin{align} \label{n}
    \sum_{z\in D_{uv}} d(z) \le |D_{uv}| \, n.
\end{align}
Combining \eqref{2n+4.}, \eqref{2n+4..} and \eqref{n}, if $G$ is not isomorphic to $K_{n,n+1}$, then we obtain
\begin{align*}
\sum_{z\in V(G)} d(z)
&= \sum_{z\in A_u \cup B_{uv} \cup A_v} d(z) + \sum_{z\in D_{uv}} d(z) + d(u) + d(v) \\
&\le (2n - 1 - |D_{uv}| - 2)n + 2n + 4 + |D_{uv}|n + 2(n + c)-|B_{uv}|\cdot\max\{0,c-3-\textbf{1}_{uv}\} \\
%&= (x - 1 - 2c + 2\cdot\mathbf{1}_{uv})n + \bigl(2n - 1 - (x - 1 - 2c + 2\cdot\mathbf{1}_{uv}) - 2\bigr)n + 2n + 4 + 2n + 2c \\
&= 2n^2 + n + 2c + 4-x\cdot\max\{0,c-3-\textbf{1}_{uv}\} .
\end{align*}
This together with the fact $x\ge2c-1$ (by Observation~\ref{1}) and $e(G)  \ge n^2 + n$ yields that
\[
n\le 2c + 4 - (2c-1)\cdot\max\{0,c-3-\textbf{1}_{uv}\}.
\]
This leads to a contradiction for $n\ge11$ unless $c=4$ and $\mathbf{1}_{uv}=1$. In the last situation, we conclude that $|B_{uv}|\ge 2c-1\ge7$ and $d(w) \le n $ for any $w \in B_{uv}$. It follows that there must exist two distinct vertices $u_1,u_2$ in $B_{uv}$ such that $d(u_1)=d(u_2)=n$. Otherwise, we have $\sum_{z\in A_u \cup B_{uv} \cup A_v} d(z)
\le(2n - 1 - |D_{uv}| - 2)n + (2n + 4)-2,$ which together with  \eqref{n} implies that $\sum_{z\in V(G)} d(z)\le2n^2+n+10$. This leads to a contradiction for $n\ge11$.
Recall that $x\le\beta-\mathbf{1}_{uv}=n+3$ and
$
|D_{uv}|=x-1-2c+2\cdot\mathbf{1}_{uv}.
$
We have $|D_{uv}|\le n-4$. This together with $d(u_1)=n$ yields that
$
|N(u_1)\cap(A_u\cup B_{uv}\cup A_v)|\ge2.
$
Thus, there exists a vertex $u_3\in N(u_1)\cap(A_u\cup B_{uv}\cup A_v)$ different from $u_2$.
Without loss of generality, we may assume that $u_3\in N(u)$.
Since $|B_{uv}|\ge7$, there also exists a vertex $u_4\in B_{uv}$ distinct from $u_1,u_2,u_3$.
Clearly, $u_1u_3uu_4vu_2$ is a path of length five with equal-degree endpoints, a contradiction.
%By Observation~\ref{1}, we have $x\ge2c-1$.
%By our assumption that $n \ge 17$, and since $c$ is an integer, it follows that $c \ge \frac{n-4}{2} \ge 7$.
%Since $c \ge 7$, by Observation~\ref{1}, we obtain $x = |B_{uv}| \ge 13$.
%Furthermore, by Lemma~\ref{dudv}, for any $w \in B_{uv}$, we have $d(w) \le 2n - (n + c) + 4 = n - c + 4$.
%Thus,
%\begin{align}\label{2n+4..}
%\sum_{z\in A_u \cup B_{uv} \cup A_v} d(z)
%\le |B_{uv}|(n - c + 4) + (2n - 1 - |D_{uv}| - 2 - |B_{uv}|)n + (2n + 4).
%\end{align}
%From \eqref{n} and \eqref{2n+4..}, we obtain
%\begin{align*}
%\sum_{z\in V(G)} d(z)
%&= \sum_{z\in A_u \cup B_{uv} \cup A_v} d(z) + \sum_{z\in D_{uv}} d(z) + d(u) + d(v)\\
%&\le (2n - 1 - |D_{uv}| - |B_{uv}| - 2)n + |B_{uv}|(n - c + 4) + 2n + 4 + |D_{uv}|n + d(u) + d(v) \\
%&= -cx + 4x + 2n^2 + n + 2c + 4.
%\end{align*}
%Since $2e(G) \ge 2(n^2 + n)$, we derive that
%\[4x + (2 - x)c - n + 4 \ge 0.\]
%Since $c \ge \frac{n - 4}{2}$ and $x \ge 13$, we have
%\[0 \le 4x + (2 - x)c - n + 4 \le 4x + (2 - x)\frac{n - 4}{2} - n + 4 = 4x - \frac{(n - 4)x}{2}.\]
%This contradicts our assumption that $n \ge 17$.
Hence, we complete the proof of Lemma~\ref{beta 1/2}.
\end{proof}

\subsection{Proof of Lemma \ref{beta n+2}} 
In this subsection, we prove Lemma \ref{beta n+2}. We first give the following easy bounds for $\beta$.
\begin{lem}\label{beta=4}
We have $5\leq \beta\leq \Delta$.
\end{lem}
\begin{proof}[\bf{Proof}]
The upper bound $\beta\le\Delta$ is immediate from the definition of $\beta$. For the lower bound, we may assume that $\beta\le4$. Then every degree greater than $4$ appears at most once. It follows that
\[
2e(G)=\sum_{v\in V(G)} d(v)\le 4\cdot5+\sum_{k=5}^{2n}k=2n^2+n+10<2n^2+2n,
\]
where the last inequality holds for $n \ge 11$. This leads to a contradiction as $e(G)\ge n^2 + n$.
\end{proof}

\begin{proof}[\bf{Proof of Lemma \ref{beta n+2}}]
Suppose for a contradiction that $\beta \le \Delta-4$ and $\Delta \ge n+4$.
Clearly, there is a unique vertex $v_0$ of degree $\Delta$. Partition $N(v_0)$ into two parts $A$ and $B$ with
\[
A := \{v\in N(v_0) \mid d(v) \ge 2n-\Delta+5\} \quad \text{and} \quad
B := \{v\in N(v_0) \mid d(v) \le 2n-\Delta+4\}.
\]
If there exist two vertices of $A$ with equal degree, then by Lemma~\ref{dudv} we  have
\[
\Delta = d(v_0) \le 2n - (2n-\Delta+5) + 4 = \Delta-1,
\]
a contradiction. Hence all vertices of $A$ have distinct degrees. The possible degree for a vertex of $A$ lies in the interval $[2n-\Delta+5,\;\Delta-1]$, which contains exactly $2\Delta-2n-5$ integers. Consequently,
$|A| \le 2\Delta-2n-5.$
It follows that
\begin{align}\label{|B|ge}
|B|=\Delta - |A| \ge 2n-\Delta+5.
\end{align}

Let $\overline{N}(v_0) = V(G)\setminus(\{v_0\}\cup N(v_0))$, and note that $|\overline{N}(v_0)| = 2n-\Delta$. We aim to bound $\sum_{v\in V(G)} d(v)$ from above. The sum is split as
\begin{align}\label{sumv}
\sum_{v\in V(G)} d(v) = \Delta + \sum_{v\in A\cup\overline{N}(v_0)} d(v) + \sum_{v\in B} d(v).
\end{align}
By the definition of $B$, we have $\sum_{v\in B} d(v) \le |B|(2n-\Delta+4)$. It suffices to bound $\sum_{v\in A\cup\overline{N}(v_0)} d(v)$, and we maximize it under the following constraints:
\begin{itemize}
    \item[(a)] all the vertices of $A$ have pairwise distinct degrees;
    \item[(b)] no two vertices in $A\cup\overline{N}(v_0)$ share a degree greater than $\beta$. % (they may share a degree equal to $\beta$).
\end{itemize}
Let $\lambda(2n+1,\Delta,\beta,|B|)$ denote the maximum possible sum of degrees for the $2n-|B|$ vertices in $A\cup\overline{N}(v_0)$. The degrees are chosen from $\{1,2,\dots,\Delta-1\}$.
By Lemma \ref{lambda}, we obtain
\begin{align}\label{sumA}
\sum_{v\in A\cup\overline{N}(v_0)} d(v)\le\lambda(2n+1, \Delta, \beta, |B|)
= \sum_{k=|B|}^{\Delta-1}k
+ \big(2n-\Delta\big)\beta
+ \binom{|B|-\beta}{2}
- \binom{(|B|-\beta)-(2n-\Delta)}{2}.
\end{align}

In what follows, we proceed our proof by distinguishing three cases according to the value of $|B|-\beta$.

\bigskip
{\bf Case 1}. $|B|-\beta \le 0$.
\medskip

By \eqref{sumv} and \eqref{sumA}, we obtain
\[
\sum_{v\in V(G)} d(v) \le \Delta + \sum_{k=|B|}^{\Delta-1} k + (2n-\Delta)\beta + |B|(2n-\Delta+4).
\]
The right‑hand side can be  simplified to a quadratic function $f_1(|B|)$ in the variable $|B|$ as follows:
\[
f_1(|B|) = -\frac12|B|^2 + \Bigl(2n-\Delta+\frac92\Bigr)|B| + (2n-\Delta)\beta + \frac12\Delta^2 + \frac12\Delta.
\]
Recall that $|B| \ge 2n-\Delta+5$ by \eqref{|B|ge}. Since the coefficient of $|B|^2$ is negative and the axis of symmetry of the parabola is given by $|B| = 2n-\Delta+\frac92$, the function decreases for $|B| \ge 2n-\Delta+5$. Thus
\[
\sum_{v\in V(G)} d(v) \le f_1(2n-\Delta+5) = 2n^2+9n-2n\Delta+\Delta^2-4\Delta + (2n-\Delta)\beta +10.
\]
This together with $2e(G) = \sum_{v\in V(G)} d(v) \ge 2n^2+2n$ implies that
\begin{align}\label{CASE1}
    (2n-\Delta)\beta - 2n\Delta +7n + \Delta^2 -4\Delta +10 \ge 0.
\end{align}
Note that $\Delta\le2n$. If $\Delta =2n$, then \eqref{CASE1} gives $n\le 10$, a contradiction. Hence, we have $2n-\Delta\ge1$. It follows from \eqref{CASE1} that
\[
\beta \ge\frac{2n\Delta-\Delta^2-7n+4\Delta-10}{2n-\Delta}= \Delta-4 + \frac{n-10}{2n-\Delta},
\]
which contradicts $\beta \le \Delta-4$ for $n\ge 11$.

\bigskip
{\bf Case 2}. $|B|-\beta \ge 2n-\Delta+1$.
\medskip

By \eqref{sumv} and \eqref{sumA}, we obtain
\begin{align*}
\sum_{v\in V(G)} d(v) 
&\le \Delta +\sum_{k=|B|}^{\Delta-1}k+\big(2n-\Delta\big)\beta
+ \binom{|B|-\beta}{2}
- \binom{(|B|-\beta)-(2n-\Delta)}{2} + |B|(2n-\Delta+4)\\
&=\Delta +\sum_{k =|B|-(2n-\Delta)}^{\Delta-1} k + |B|(2n-\Delta+4).
\end{align*}
This rightmost expression can be simplified to a quadratic function $f_2(|B|)$ in $|B|$ as follows:
\[
f_2(|B|) = -\frac12|B|^2 + \Bigl(4n-2\Delta+\frac92\Bigr)|B| -2n^2+2n\Delta -n +\Delta.
\]
Clearly, the maximum of $f_2(|B|)$ occurs at $|B| = 4n-2\Delta+\frac92$. Since $|B|$ is an integer, we have
\begin{align}\label{case21}
    \sum_{v\in V(G)} d(v) \le f_2(4n-2\Delta+4) = 2\Delta^2 - (6n+8)\Delta + 6n^2+17n+10.
\end{align}
Using $2e(G) \ge 2n^2+2n$ gives that
\begin{align*}\label{case22}
    2\Delta^2 - (6n+8)\Delta + 4n^2+15n+10 \ge 0.
\end{align*}
Solving the quadratic inequality in $\Delta$ yields
\[
\Delta \le \frac{3n+4 - \sqrt{n^2-6n-4}}{2} \quad\text{or}\quad \Delta \ge \frac{3n+4 + \sqrt{n^2-6n-4}}{2}.
\]
The first upper bound is less than $n+4$ for $n\ge 11$, contradicting $\Delta \ge n+4$; the second lower bound exceeds $2n$ for $n\ge 11$, which is impossible.

\bigskip
{\bf Case 3}. $1\le |B|-\beta \le 2n-\Delta$.
\medskip

By \eqref{sumv} and \eqref{sumA}, we obtain
\begin{align}\label{case3}
\sum_{v\in V(G)} d(v) &\le \Delta +\sum_{k=|B|}^{\Delta-1}k+\big(2n-\Delta\big)\beta
+ \binom{|B|-\beta}{2}
 + |B|(2n-\Delta+4)\notag
 \\&=
 \Delta + \sum_{k=\beta+1}^{\Delta-1} k + (2n-\Delta+\beta+1)\beta + |B|\bigl(2n-\Delta+4-\beta\bigr).
\end{align}
In what follows, we consider two subcases depending on the sign of $2n-\Delta+4-\beta$.

\bigskip
{\bf Subcase 3a}. $2n-\Delta+4-\beta \ge 0$.
\medskip

Note that the coefficient of $|B|$ in \eqref{case3} is non‑negative. Using $|B| \le \Delta$, we obtain
\[
\sum_{v\in V(G)} d(v) \le \Delta + \sum_{k=\beta+1}^{\Delta-1} k + (2n-\Delta+\beta+1)\beta + \Delta(2n-\Delta+4-\beta).
\]
Simplifying the rightmost expression gives a quadratic function $f_3(\beta)$ in $\beta$ as follows:
\[
f_3(\beta) = \frac12\beta^2 + \Bigl(2n-2\Delta+\frac12\Bigr)\beta + 2n\Delta -\frac12\Delta^2 +\frac92\Delta.
\]
By Lemma~\ref{beta=4}, $\beta \ge 5$, and the subcase condition implies $\beta \le 2n-\Delta+4$. Thus $\beta$ lies in $[5,\,2n-\Delta+4]$. Since $f_3$ is convex, its maximum on this interval occurs at an endpoint. Evaluating $f_3(5)$ gives
\[
f_3(5) = -\frac12\Delta^2 + \Bigl(2n-\frac{11}2\Bigr)\Delta + 10n+15.
\]
Using $2e(G) \ge 2n^2+2n$ gives
\begin{align*}        
-\frac{1}{2}\Delta^2 + (2n-\frac{11}{2}) \Delta-2n^2+8n+15\ge 0.
\end{align*}
View the expression above as a quadratic polynomial in the variable \(\Delta\).
Then its discriminant
$
\left(2n - \frac{11}{2}\right)^2 - 2(2n^2 - 8n - 15) = -6n + \frac{241}{4}
$
must be non-negative; however, this leads to a contradiction for $n\ge 11$. Hence, the maximum is attained at the other endpoint and
\[
f_3(\beta)\le f_3(2n-\Delta+4) = 2\Delta^2 - (6n+8)\Delta + 6n^2+17n+10,
\]
which is exactly the expression obtained in \eqref{case21}. Arguments similar to those in Case~2 yield a contradiction.

\bigskip
{\bf Subcase 3b}. \(2n-\Delta+4-\beta < 0\).
\medskip

Note that the coefficient of \(|B|\) in \eqref{case3} is negative.  Since \(|B|\ge \beta+1\), we have 
\[
\sum_{v\in V(G)} d(v) \le \Delta + \sum_{k=\beta+1}^{\Delta-1} k + (2n-\Delta+\beta+1)\beta + (\beta+1)(2n-\Delta+4-\beta).
\]
Simplifying the rightmost expression gives a quadratic function $f_4(\beta)$ in $\beta$ as follows:
%Simplifying gives another quadratic in \(\beta\):
\begin{align*}\label{CASE3B}
    f_4(\beta) = -\frac12\beta^2 + \Bigl(4n-2\Delta+\frac72\Bigr)\beta + \frac12\Delta^2 -\frac12\Delta + 2n +4.
\end{align*}
Since \(\beta\) is an integer, the maximum value of $f_4(\beta)$ is achieved at \(\beta = 4n-2\Delta+3\). Then 
\[
f_4(\beta)\le f_4(4n-2\Delta+3) = \frac52\Delta^2 - \Bigl(8n+\frac{15}{2}\Bigr)\Delta + 8n^2+16n+10.
\]
Using \(2e(G)\ge2n^2+2n\), we obtain
\[
\frac52\Delta^2 - \Bigl(8n+\frac{15}{2}\Bigr)\Delta + 6n^2+14n+10 \ge 0.
\]
Solving this quadratic inequality in \(\Delta\) yields either
\[
\Delta \ge \frac{16n+15 + \sqrt{16n^2-80n-175}}{10}
\quad\text{or}\quad
\Delta \le \frac{16n+15 - \sqrt{16n^2-80n-175}}{10}.
\]
The first bound implies that $\Delta>2n$ for $n\ge11$, which is impossible. The second one shows that $\Delta<6n/5+3$ for  $n\ge 11$, implying that $\Delta-4 < 4n-2\Delta+3$ for $n\ge 11$. This together with $\beta \le \Delta-4$ deduces that
\[
f_4(\beta)\le f_4(\Delta-4) = -2\Delta^2 + (4n+15)\Delta -14n -18.
\]
Using $2e(G) \ge 2n^2+2n$ gives
\begin{align*}        
-2\Delta^2 + (4n+15)\Delta -2n^2 -16n -18 \ge 0.
\end{align*}
View the expression above as a quadratic polynomial in the variable \(\Delta\).
Then its discriminant
\((4n+15)^2 - 8(2n^2+16n+18) = -8n +81\)
must be non-negative; however, this leads to a contradiction for $n\ge 11$.

Combining the above three cases, we conclude that our initial assumption is false. Hence, either \(\beta \ge \Delta-3\) or \(\Delta \le n+3\), completing the proof of Lemma \ref{beta n+2}.
\end{proof}

\section{Concluding remarks}
In this paper, we show that for all $n \ge 11$, $K_{n,n+1}$ is the unique $(2n+1)$-vertex graph with at least $n^2+n$ edges that avoids two equal-degree vertices joined by a path of length five. Our method can also be extended to graphs with an even number of vertices, yielding the following theorem. We leave the proof of this theorem in the appendix to avoid redundancy, as it is essentially identical to that of Theorem~\ref{LZ}.
\begin{thm}\label{even}
Let $n\ge 13$ be an integer. The unique \(2n\)-vertex graph with at least \(n^{2}-1\) edges,
which does not contain two vertices of the same degree joined by a path of length five, is the
complete bipartite graph \(K_{n-1,n + 1}\).
\end{thm}
%Since the proof is essentially identical to that of Theorem~\ref{LZ}, we leave it to the appendix to avoid redundancy.

Chen and Ma \cite{CM2025} also posed the following problem on paths of even length with equal-degree endpoints, and asserted that the behavior of \(p_\ell(N)\) differs significantly between odd and even $\ell$.
\begin{prob}[Chen and Ma \cite{CM2025}]\label{Prob-CM}
Determine the exact value of \(p_\ell(2n)\) for all even \(\ell\) and sufficiently large \(n\).
\end{prob}

In upcoming work, we show that \(p_4(2n)\le(1-\epsilon+o(1))n^2\) for some constant $\epsilon\in(0,1/2]$ and sufficiently large $n$; in particular, for $n\ge400$, among all $2n$-vertex graphs with at least $(n^2+n)/2$ edges, the half graph is the unique graph that contains two vertices of equal degree at least $n$ and admits no path of length four between any two vertices of equal degree.

%the unique \(2n\)-vertex graph with at least \(\frac{n(n+1)}{2}\) edges, which contains two vertices of equal degree at least \(n\) and does not contain two vertices of equal degree connected by a path of length four, is the half graph. \(p_4(2n)=(n^2+n)/2\) for sufficiently large $n$ if the equal-degree joined by the path of length four is at least \(n\), and the extremal graph is unique.

\section{Appendix: Proof of Theorem \ref{even}}
In this section, we complete the proof of Theorem~\ref{even}. Let \(n \ge 13\) be an integer and let \(G\) be a graph on \(2n\) vertices with at least \(n^2-1\) edges. Assume that \(G\) contains no two vertices of equal degree joined by a path of length five. Our goal is to show that \(G\) must be the complete bipartite graph \(K_{n-1,n+1}\).

We first establish the following two lemmas. Recall that \(\beta\) is the largest integer such that \(G\) contains two vertices of degree \(\beta\).
\begin{lem}\label{ebeta 1/2}
   We have \(\beta \leq n + 1\). Moreover, if \(\beta = n + 1\), then \(G\) is isomorphic to \(K_{n-1,n+1}\).
\end{lem}
\begin{lem}\label{ebeta n+2}
We have either $\beta \geq \Delta-3$ or $\Delta\leq n+3$.
\end{lem}
We now give a proof of Theorem \ref{even}, postponing the proofs of Lemmas~\ref{ebeta 1/2} and \ref{ebeta n+2} to the next two subsections.
\begin{proof}[{\bf Proof of Theorem \ref{even}}]
Lemmas~\ref{ebeta 1/2} and \ref{ebeta n+2} imply that either (i) $\beta = n+1$ and $G$ is isomorphic to $K_{n-1,n+1}$, or (ii) $\beta \le n$ and $\Delta \le n+3$. In what follows, we may assume that $G$ is not isomorphic to $K_{n-1,n+1}$.
\begin{claim}\label{Vtx-Deg-n}
The following results hold:
\begin{itemize}
    \item[$(\mathrm{i})$] There are at least $2n-11$ vertices of degree $n$ in $G$.
    \item[$(\mathrm{ii})$] There exists a vertex $w$ such that $d(w)\ge n+2$.
\end{itemize}
\end{claim}
\begin{proof}
(i) Suppose for a contradiction that there are at most $2n-12$ vertices of degree $n$ in $G$. Note that $\beta \le n$. Then
\[
\sum_{v\in V(G)} d(v) \le \sum_{k=n+1}^{n+3} k + (2n-12)n + \bigl(2n-3-(2n-12)\bigr)(n-1) = 2n^2 - 3,
\]
which contradicts $e(G)\ge n^2-1$.

(ii) Suppose that $\Delta\le n+1$. We first show that $G$ has at most $4$ vertices whose degree is not equal to $n$.
Suppose, to the contrary, that there exist $5$ vertices with degree different from $n$. Note that $\beta \le n$. Then
\[
\sum_{v\in V(G)} d(v) \le n(2n-5)+(n+1)+4(n-1)=2n^2-3,
\]
which contradicts $e(G)\ge n^2-1$. Let $u_1$ be a vertex with $d(u_1)=n$.
Since $n\ge13$ and at most $4$ vertices have degree not equal to $n$,
there exist distinct vertices $u_2,u_3\in N(u_1)$ with equal degree $n$.
Similarly, we can find vertices
$u_4\in N(u_2)$, $u_5\in N(u_3)$, $u_6\in N(u_5)$ with equal degree $n$,
all of which are different from $u_1,u_2,u_3$.
Then $u_4u_2u_1u_3u_5u_6$ is a path of length $5$ with equal-degree endpoints, a contradiction.
Thus there exists a vertex $w$ such that $d(w)\ge n+2$.
\end{proof}

Let $W = \{v\in V(G) \mid d(v)=n\}$. By Claim \ref{Vtx-Deg-n}, we have $|W|\ge2n-11$ and $d(w)\ge n+2$ for some $w$ in $G$.
It follows that $|N(w)\cap W|\ge n-8\ge5$ as $n\ge13$.
Choose $u_1\in N(w)\cap W$.
By Observation~\ref{1}, there exists a vertex $u_2\in N(u_1)\cap N(w)$.
Since $|N(w)\cap W|\ge5$, we may choose three distinct vertices
$u_3,u_4,u_5\in \bigl(N(u_1)\cap N(w)\bigr)\setminus\{u_1,u_2\}$.
If $|N(u_3)\cap N(u_4)|\ge4$, then we can choose a vertex
$u_6\in N(u_3)\cap N(u_4)$ different from $u_1,u_2,w$,
and $u_1u_2wu_3u_6u_4$ is a path of length $5$ with endpoints of equal degree, a contradiction.
Thus $|N(u_3)\cap N(u_4)|\le3$.
By Observation~\ref{1}, we have $|D_{u_3u_4}|\le3$.
Since $n\ge13$, we obtain
\(
|N(u_5)\setminus\bigl(D_{u_3u_4}\cup\{u_3,u_4\}\bigr)|\ge n-5\ge8.
\)
This implies that either $|N(u_5)\cap N(u_3)|\ge4$ or $|N(u_5)\cap N(u_4)|\ge4$.
Without loss of generality, we may assume that $|N(u_5)\cap N(u_3)|\ge4$.
Then there exists a vertex $u_7\in N(u_5)\cap N(u_3)$ different from $w,u_1,u_2$,
and $u_1u_2wu_3u_7u_5$ is a path of length $5$ with endpoints of equal degree, a contradiction. Thus, we complete the proof of Theorem~\ref{even}.
\end{proof}

\subsection{Proof of Lemma \ref{ebeta 1/2}}
In this subsection, we give a proof of Lemma \ref{ebeta 1/2}.
\begin{proof}[\bf{Proof of Lemma \ref{ebeta 1/2}}]  
Suppose for contradiction that \(n+c=\beta \geq n + 1\). Let \(u\) and \(v\) be two vertices in \(G\) with \(d(u) = d(v) = \beta\), and let \(\textbf{1}_{uv} = 1\) if \(uv \in E(G)\) and \(\textbf{1}_{uv} = 0\) if \(uv \notin E(G)\). Let \(|B_{uv}| = x\). Then, \(|A_u| =|A_v|=n+c-x-\textbf{1}_{uv}\) and \(|D_{uv}| = 2n-2-|B_{uv}| -|A_u|-|A_v|=x-2-2c+2\textbf{1}_{uv}\). Note that since $c \ge 1$, we have $|D_{uv}| = x - 2 - 2c + 2\cdot\mathbf{1}_{uv} < x = |B_{uv}|$, hence $|D_{uv}| < |B_{uv}|$.

\begin{claim}\label{eD2N+1}
   Let $|D_{uv}| \neq 0$ and suppose there exists a vertex $w \in D_{uv}$ with $d(w) \ge n$. Then for every vertex $w_1 \in A_u \cup A_v \cup B_{uv}$, either (i) $d(w)+d(w_1) \le 2n-1$, or (ii) $d(w)+d(w_1)=2n$ and $ww_1 \in E(G)$.
\end{claim}
\begin{proof}
   Assume, to the contrary, that there exists some $w_1 \in A_u \cup A_v \cup B_{uv}$ such that $d(w)+d(w_1) \ge 2n+1$, or $d(w)+d(w_1)=2n$ but $ww_1 \notin E(G)$. Without loss of generality, we may suppose $w_1 \in N(u)$. Applying Observation \ref{1} to the pair $(w,w_1)$ and using the above condition, we obtain $|N(w)\cap N(w_1)| \ge 1$; hence there exists a vertex $w_2 \in N(w)\cap N(w_1)$. Since $w\in D_{uv}$, we have $wu,wv\notin E(G)$. Now apply Observation \ref{1} to $(w,v)$; from $d(w)\ge n$ and $d(v)=\beta\ge n+1$ we get $|N(w)\cap N(v)| \ge 3$. Consequently we can choose a vertex $w_3\in N(w)\cap N(v)$ distinct from $w_1,w_2,u$. Then $u w_1 w_2 w w_3 v$ is a path of length five whose endpoints $u$ and $v$ have equal degree, contradicting the assumption on $G$.
\end{proof}
\begin{claim}\label{e2.9}
  If $|D_{uv}|\neq 0$, then for any vertex $w\in D_{uv}$, we have $d(w)\le n+1$.
\end{claim}
\begin{proof}
 Choose $w$ to be a vertex of maximum degree in $D_{uv}$. Assume, to the contrary, that $d(w) \ge n+2$. By Claim \ref{eD2N+1}, for every $w_1\in A_u\cup B_{uv}\cup A_v$, we have $d(w)+d(w_1)\le 2n$, and consequently $d(w_1)\le n-2$ because $d(w)\ge n+2$. Moreover, for any $z\in D_{uv}$ and any $w_1\in A_u\cup B_{uv}\cup A_v$, we have $d(z)+d(w_1)\le d(w)+d(w_1)\le 2n$.

Recall that $|D_{uv}| = x-2-2c+2\cdot\mathbf{1}_{uv}$ and $|A_u\cup B_{uv}\cup A_v| = 2n-2-|D_{uv}|$. Note that $|D_{uv}| < |B_{uv}| = x$ because $c\ge 1$. Also, we have $|D_{uv}| \le |A_u\cup B_{uv}\cup A_v|$. Thus we can pair each vertex of $D_{uv}$ with a distinct vertex of $A_u\cup B_{uv}\cup A_v$, leaving $|A_u\cup B_{uv}\cup A_v| - |D_{uv}|$ vertices unpaired. For each paired couple $(z,w_1)$, we have $d(z)+d(w_1)\le 2n$, and each unpaired vertex has degree at most $n-2$. Let $D'$ be the set of vertices in $A_u\cup B_{uv}\cup A_v$ paired with $D_{uv}$. Therefore,
\begin{align*}    
\sum_{z\in V(G)} d(z) 
&\le \sum_{z\in D'\cup D_{uv}} d(z) + \sum_{z\in (A_u\cup B_{uv}\cup A_v)\setminus D'} d(z) + d(u)+d(v) \\ &\le |D_{uv}|(2n) + (|A_u\cup B_{uv}\cup A_v| - |D_{uv}|)(n-2) + d(u)+d(v)\\    
&= (x-2-2c+2\cdot\mathbf{1}_{uv})(2n)+(2n-2-2(x-2-2c+2\cdot\mathbf{1}_{uv}))(n-2)+2(n+c)\\    
&= 2n^2 - 4n - 6c + 4x + 8\cdot\mathbf{1}_{uv} - 4. \\    
&\le 2n^2-2c+4\cdot\mathbf{1}_{uv}-4.
\end{align*}
The last inequality holds as $|B_{uv}|=x \le n + c - \mathbf{1}_{uv}$.
Since $\sum_{v\in V(G)} d(v) = 2e(G) \ge 2(n^2-1)$, it follows that
\(
4\cdot\mathbf{1}_{uv} - 2c - 2\geq 0.
\)   
Since $\mathbf{1}_{uv} \in \{0,1\}$ and $c \ge 1$,
we conclude that $c = 1$ and $\mathbf{1}_{uv} = 1$.
Substituting these equalities into the previous estimates,
we find that every intermediate inequality must hold with equality.
Hence we obtain $x = n + c - \mathbf{1}_{uv} = n$
and consequently $|D_{uv}| = x - 2 - 2c + 2\cdot\mathbf{1}_{uv} = n - 2$.
Now $|D_{uv}|=n-2\ge 2$. Because $w$ is a vertex of maximum degree in $D_{uv}$ and $d(w)\ge n+2$, if every vertex in $D_{uv}$ had degree $d(w)$, then $d(w)$ would appear at least twice, contradicting the definition of $\beta$. Thus there exists a vertex $z\in D_{uv}$ distinct from $w$ such that $d(z) < d(w)$. For this $z$, we have $d(z)+d(w_1) < d(w)+d(w_1) \le 2n$ for every $w_1\in A_u\cup B_{uv}\cup A_v$. Consequently, in the pairing, the sum for the couple containing $z$ is strictly less than $2n$, contradicting the equality required above. This completes the proof.
\end{proof}
\begin{claim}\label{e2.10}
    If $|D_{uv}| \neq 0$, then either $d(w) \le n-1$ for every vertex $w \in D_{uv}$, or $G$ is isomorphic to $K_{n-1,n+1}$.
\end{claim}
\begin{proof}
    By Claim~\ref{e2.9}, every vertex in $D_{uv}$ has degree at most $n+1$.
Suppose for contradiction that there exists a vertex $w \in D_{uv}$ with $d(w) = n + c_1$ where $c_1 \in \{0, 1\}$. (Otherwise, all vertices in $D_{uv}$ have degree at most $n-1$, and we are done.)
Apply Claim~\ref{eD2N+1} to $w$ and any $w_1 \in A_u \cup B_{uv} \cup A_v$.
Since $d(w) = n + c_1$, we have either $d(w_1) \le n - c_1 - 1$, or $d(w_1) = n - c_1$ and $ww_1 \in E(G)$.
In particular, $d(w_1) \le n - c_1$ for all $w_1 \in A_u \cup B_{uv} \cup A_v$.

\bigskip
{\bf Case 1}: There exists $w_1\in A_u\cup B_{uv}\cup A_v$ with $d(w_1)=n-c_1$.
\medskip

We first show that $w_1$ has no neighbor in $A_u\cup B_{uv}\cup A_v$. Suppose, for contradiction, that there exists $w_2\in N(w_1)\cap (A_u\cup B_{uv}\cup A_v)$. Without loss of generality, assume $w_2\in N(u)$. Since $w\in D_{uv}$, we have $wu,wv\notin E(G)$. Applying Observation \ref{1} to $(w,v)$ yields $|N(w)\cap N(v)|\ge 3$. Hence we can pick $w_3\in N(w)\cap N(v)$ distinct from $w_1,w_2$. Then $u w_2 w_1 w w_3 v$ is a path of length five joining the equal-degree vertices $u$ and $v$, a contradiction. Thus $N(w_1)\cap (A_u\cup B_{uv}\cup A_v)=\emptyset$.

Since $d(w_1)=n-c_1$ and $N(w_1)\cap (A_u\cup B_{uv}\cup A_v)=\emptyset$, we must have $|N(w_1)\cap D_{uv}|\ge n-c_1-2$. Consequently $|D_{uv}|\ge n-c_1-2$. On the other hand, from the definition of $D_{uv}$ and $x=|B_{uv}|\le n+c-\mathbf{1}_{uv}$, we have
\[
|D_{uv}| = x-2-2c+2\cdot\mathbf{1}_{uv} \le n-c-2+\mathbf{1}_{uv}.
\]

If $\mathbf{1}_{uv} = 0$, then $|D_{uv}| \le n - c - 2 \le n - 3$.
Combined with $|D_{uv}| \ge n - c_1 - 2$, we obtain $|D_{uv}| = n - 3$, $c = 1$, and $c_1 = 1$. Since $c=1$ and $|D_{uv}|=n-3$, we obtain $|B_{uv}|=x=n+1$ and $|A_u|=|A_v|=0$. If there exists a vertex in $ B_{uv}$ with degree at most $n-2$, then since every vertex in $D_{uv}$ has degree at most $n+1$, every vertex in $ B_{uv}$ has degree at most $n-1$, and $d(u)=d(v)=n+1$, this contradicts our assumption that $G$ has at least $n^2-1$ edges.
This implies that all $x = n+1$ vertices in $B_{uv}$ have degree $n-1$, so the vertices in $B_{uv}$ are pairwise nonadjacent.
A straightforward verification shows that $G$ is exactly $K_{n-1,n+1}$ in this subcase. 

If $\mathbf{1}_{uv} = 1$ and $c_1 = 0$, then $|D_{uv}| \le n - c - 1$.
Together with $|D_{uv}| \ge n - c_1 - 2 = n - 2$, this implies $c = 1$ and $|D_{uv}| =x-2c= n - 2$.
If $\mathbf{1}_{uv} = 1$ and $c_1 = 1$, then $|D_{uv}| \le n - c - 1$.
Together with $|D_{uv}| \ge n - c_1 - 2$, this implies $1 \le c \le 2$. 
We now argue that no two vertices in $D_{uv}$ can have degree $n+c_1$. Indeed, if $u_1,u_2\in D_{uv}$ both had degree $n+c_1$, applying  Observation \ref{1} to the pairs $(u,u_1)$ and $(v,u_2)$ yields at least three common neighbors for each pair. So we may assume that $w_1 \in N(u_1) \cap N(u)$ and $w_2 \in N(u_2) \cap N(v) \setminus \{w_1\}$, leading to a path $u_1 w_1 u v w_2 u_2$ of length five with equal-degree endpoints $u_1,u_2$, which is a contradiction. Therefore at most one vertex in $D_{uv}$ has degree $n+c_1$, and all others have degree at most $n+c_1-1$. Then
\begin{align}\notag   
\sum_{z\in V(G)} d(z) &\le (|D_{uv}|-1)(n+c_1-1) + (n+c_1) + (2n-2-|D_{uv}|)(n-c_1)+ d(u)+d(v)\\ 
&=(x-2c-1)(n+c_1-1)+(n+c_1)+(2n-2-(x-2c))(n-c_1)+2(n+c)\notag\\    
&=2n^2+4c+1+(2x-2n-4c+2)c_1-x.\label{ec1}
\end{align}
If $\mathbf{1}_{uv} = 1$ and $c_1 = 0$, then \eqref{ec1} implies
\[
\sum_{v \in V(G)} d(v) \le 2n^2 + 4c + 1 - x = 2n^2 + 5 - n.
\]
The last equality holds since $c = 1$ and $|D_{uv}| = x - 2c = n - 2$.
Since we assume $n \ge 13$, this contradicts our assumption that $G$ has at least $n^2 - 1$ edges. 

If $\mathbf{1}_{uv} = 1$ and $c_1 = 1$, then \eqref{ec1} implies
\[
\sum_{v \in V(G)} d(v) \le 2n^2 - 2n + 3 + x \le 2n^2 - n + c + 3 - \mathbf{1}_{uv}.
\]
The last inequality holds since $x \le n + c - \mathbf{1}_{uv}$.
As $1 \le c \le 2$ and we assume $n \ge 13$, this contradicts our assumption that $G$ has at least $n^2 - 1$ edges.

\bigskip
{\bf Case 2}: For every vertex $w \in A_u\cup B_{uv}\cup A_v$, $d(w) \le n-c_1-1$.
\medskip

Since every vertex in $D_{uv}$ has degree at most $n + c_1$
and every vertex in $A_u \cup B_{uv} \cup A_v$ has degree at most $n - c_1-1$,
we obtain
\begin{align}
    \sum_{z\in V(G)} d(z) &\le |D_{uv}|(n+c_1) + (2n-2-|D_{uv}|)(n-c_1-1) + d(u)+d(v)\notag\\
    &=(x-2-2c+2\cdot\mathbf{1}_{uv})(n+c_1)+(2n-2-(x-2-2c+2\cdot\mathbf{1}_{uv}))(n-c_1-1)+2(n+c)\notag\\
    &=2n^2-2n+x+2\cdot\mathbf{1}_{uv}+(2x+4\cdot\mathbf{1}_{uv}-2n-4c-2)c_1\notag\\
    &\le 2n^2-n+c+\mathbf{1}_{uv}+(2\cdot\mathbf{1}_{uv}-2c-2)c_1.\label{ecc}
\end{align}
The last inequality holds since $x \le n + c - \mathbf{1}_{uv}$. If $c_1 = 1$, then \eqref{ecc} implies
\[
\sum_{v \in V(G)} d(v) \le 2n^2 - n - c + 3\cdot\mathbf{1}_{uv} - 2.
\]
Since $n \ge 13$, this contradicts our assumption that $e(G)\ge n^2 - 1$.

If $c_1 = 0$, then \eqref{ecc} implies
\[
\sum_{v \in V(G)} d(v) \le 2n^2 - n + c + \mathbf{1}_{uv}.
\]
Since $G$ has at least $n^2 - 1$ edges, we further obtain $c + \mathbf{1}_{uv} + 2 - n \ge 0$, i.e., $c \ge n - 3$.
Then by Lemma~\ref{dudv}, for any $w \in B_{uv}$, we have $d(w) \le 2n - (n + c) + 3 = n - c + 3$.
Since every vertex in $D_{uv}$ has degree at most $n + c_1 = n$, every vertex in $A_u \cup A_v$ has degree at most $n - c_1 - 1 = n - 1$, and every vertex in $B_{uv}$ has degree at most $n - c + 3$, we have
\begin{align*}
\sum_{z \in V(G)} d(z)
&\le |D_{uv}|  n + \bigl(2n - 2 - (|D_{uv}| + |B_{uv}|)\bigr)  (n-1) + |B_{uv}|(n - c + 3) + d(u) + d(v) \\
&= \bigl(x - 2 - 2c + 2\cdot\mathbf{1}_{uv}\bigr)n + \bigl(2n  - (2x  - 2c + 2\cdot\mathbf{1}_{uv})\bigr)(n-1) + x(n - c + 3) + 2(n + c) \\
&= 2n^2 - 2n - cx + 5x + 2\cdot\mathbf{1}_{uv} \\
&\le 2n^2 - 2n + 2\cdot\mathbf{1}_{uv}.
\end{align*}
The last inequality holds since $c \ge n - 3$ and $n \ge 13$.
This contradicts our assumption that $e(G)\ge n^2 - 1$.
%Thus Case 2 cannot occur.
\end{proof}

\begin{claim}\label{e2.11}
    For any two vertices $w_1, w_2 \in A_u \cup B_{uv} \cup A_v$ with $d(w_1) \ge n $ and $d(w_2) \ge n $, we have $w_1w_2 \in E(G)$.
\end{claim}
\begin{proof}
Suppose for a contradiction that $w_1w_2 \notin E(G)$, and assume that $w_1 \in N(u)$.

\textbf{Case 1}: $(N(w_1) \cap N(w_2)) \setminus \{u, v\} \neq \emptyset$.

Suppose there exists a vertex $w \in (N(w_1) \cap N(w_2)) \setminus \{u, v\}$.
Then $w_2v \in E(G)$; otherwise, by Observation \ref{1}, we have $|N(w_2) \cap N(v)| \ge 3$.
Since $w_1w_2 \notin E(G)$, there exists a vertex $u_1 \in (N(w_2) \cap N(v)) \setminus \{u, w\}$.
Then $u w_1 w w_2 u_1 v$ is a path of length five whose endpoints $u$ and $v$ have the same degree, a contradiction.

We next show that $w_1v \in E(G)$ and $w_2u \in E(G)$.
If $w_1v \notin E(G)$ or $w_2u \notin E(G)$, then without loss of generality, assume $w_1v \notin E(G)$.
By Observation \ref{1}, we have $|N(w_1) \cap N(v)| \ge 3$.
Since $w_1w_2 \notin E(G)$, there exists a vertex $u_1 \in (N(w_2) \cap N(v)) \setminus \{u, w\}$, which implies $w_2u \notin E(G)$; otherwise, $v u_1 w_1 w w_2 u$ is a path of length five whose endpoints $u$ and $v$ have the same degree, a contradiction.
Since $w_1v \notin E(G)$, by Observation \ref{1}, there exists a vertex $u_1 \in (N(w_1) \cap N(u)) \setminus \{v, w_2\}$ (possibly $u_1 = w$).
Since $w_1v$, $w_2u$, and $w_1w_2$ are all not in $E(G)$, by Observation \ref{1}, there exists a vertex $u_2 \in (N(w_1) \cap N(w_2)) \setminus \{u, v, u_1\}$.
Then $u u_1 w_1 u_2 w_2 v$ is a path of length five whose endpoints $u$ and $v$ have the same degree, a contradiction.
This completes the proof that $w_1v \in E(G)$ and $w_2u \in E(G)$.

By Observation \ref{1} and $d(v), d(w_2) \ge n + 1$, we get
$|N(v) \cap N(w_2)| \ge 2.$
Since $w_1w_2 \notin E(G)$, there exists a vertex $u_1 \in (N(v) \cap N(w_2)) \setminus \{u, w_1\}$.
By Observation \ref{1} and $d(w_1), d(w_2) \ge n + 1$, we have $|N(w_1) \cap N(w_2)| \ge 4$.
Thus, there exists a vertex $u_2 \in (N(w_1) \cap N(w_2)) \setminus \{u, v, u_1\}$.
Then $u w_1 u_2 w_2 u_1 v$ is a path of length five whose endpoints $u$ and $v$ have equal degree, a contradiction.
\end{proof}
\begin{claim}\label{e2.15}
    The number of vertices in $A_u \cup B_{uv} \cup A_v$ with degree at least $n + 1$ is at most two.
\end{claim}
\begin{proof}
   Suppose, for contradiction, that there exist three distinct vertices $w_1,w_2,w_3\in A_u\cup B_{uv}\cup A_v$ with $d(w_i)\ge n+1$ for $i=1,2,3$. By Claim \ref{e2.11}, they are pairwise adjacent. Without loss of generality, assume $w_1\in N(u)$.

We first show $w_2 \in N(v)$.
Assume $w_2v \notin E(G)$.
Applying Observation \ref{1} to $(w_2, v)$ yields $|N(w_2) \cap N(v)| \ge 4$.
Thus, there exists a vertex $u_1 \in (N(w_2) \cap N(v)) \setminus \{w_1, w_3, u\}$.
Then $u w_1 w_3 w_2 u_1 v$ is a path of length five with endpoints $u$ and $v$ of equal degree, a contradiction.
By symmetry, $w_3v \in E(G)$.

Since $\{w_2, w_3\} \subseteq N(v)$, suppose there exists a vertex $w \in N(w_1) \cap N(w_2) \setminus \{w_3, v, u\}$. Then $uw_1ww_2w_3v$ is a path of length five with $d(u) = d(v)$, which is a contradiction. Thus, $N(w_1) \cap N(w_2) \subseteq \{w_3, v, u\}$. By Observation \ref{1}, we have $|D_{w_1w_2}| \le 1$. 
Since $d(w_3) \ge n + 1$, it follows that
\(|N(w_3) \setminus (D_{w_1w_2} \cup \{w_1, w_2\})| \ge n - 2.\)
Therefore, together with our assumption that $n \ge 13$, this implies that either $|N(w_3) \cap N(w_1)| \ge \frac{n - 2}{2} > 5$ or $|N(w_3) \cap N(w_2)| \ge \frac{n - 2}{2} > 5$.  
Without loss of generality, assume that $|N(w_3) \cap N(w_1)| \ge 8$. Then there exists a vertex $u_1 \in N(w_3) \cap N(w_1)$ distinct from $w_2, u, v$. Thus, $uw_2w_3u_1w_1v$ is a path of length five with $d(u) = d(v)$, a contradiction.  
This completes the proof.
\end{proof}
\begin{claim}\label{e2n+4}
    For any vertices $w_1, w_2 \in A_u \cup B_{uv} \cup A_v$, we have $d(w_1) + d(w_2) \le 2n + 3$.
\end{claim}
\begin{proof}
Suppose there exist two vertices $w_1, w_2 \in A_u \cup B_{uv} \cup A_v$ such that $d(w_1) + d(w_2) \ge 2n + 4$. Since $d(w_1)$ and $d(w_2)$ are integers, at least one of $d(w_1)$ and $d(w_2)$ is at least $n + 2$. Without loss of generality, assume $d(w_1) \ge n + 2$. 

Since $w_2 \in A_u \cup B_{uv} \cup A_v$, we assume without loss of generality that $w_2 \in N(u)$. As $d(w_1) \ge n + 2$ and $d(v) \ge n + 1$, by Observation \ref{1}, we have $|N(w_1) \cap N(v)| \ge 3$. Thus, there exists a vertex $w$ distinct from $w_2, u$ such that $w \in N(w_1) \cap N(v)$. 
Since $d(w_1) + d(w_2) \ge 2n + 4$, by Observation \ref{1}, we have $|N(w_1) \cap N(w_2)| \ge 4$. Therefore, there exists a vertex $u_1$ distinct from $u, v, w$ such that $u_1 \in N(w_1) \cap N(w_2)$. Then $uw_2u_1w_1wv$ is a path of length five with endpoints $u$ and $v$ of equal degree, a contradiction.
\end{proof}
\begin{claim}\label{enn}
    The number of vertices in $A_u \cup B_{uv} \cup A_v$ with degree exactly $n$ is at most two.
\end{claim}
\begin{proof}
    Suppose, for contradiction, that there exist three distinct vertices $w_1,w_2,w_3\in A_u\cup B_{uv}\cup A_v$ with $d(w_i)= n$ for $i=1,2,3$.
We first show that there exist $j,k\in\{1,2,3\}$ such that $|N(w_j)\cap N(w_k)|\ge 4$.
If $|N(w_1)\cap N(w_2)|\le 3$, then since $d(w_1)=d(w_2)=n$, by Observation \ref{1} we have $|D_{w_1w_2}|\le 3$.
Since $d(w_3)=n$, it follows that
\(|N(w_3) \setminus (D_{w_1w_2} \cup \{w_1, w_2\})| \ge n - 5.\)
Therefore, together with our assumption that $n \ge 13$, this implies that either $|N(w_3) \cap N(w_1)| \ge \frac{n - 5}{2} \ge 4$ or $|N(w_3) \cap N(w_2)| \ge \frac{n - 5}{2} \ge 4$.
This proves that there exist $j,k\in\{1,2,3\}$ such that $|N(w_j)\cap N(w_k)|\ge 4$.

Without loss of generality, assume $|N(w_1)\cap N(w_2)|\ge 4$.
Then there exist two distinct vertices $u_1,u_2\notin\{u,v\}$ such that $u_1,u_2\in N(w_1)\cap N(w_2)$.
Note that $w_1,w_2\in A_u\cup B_{uv}\cup A_v$, so without loss of generality assume $w_1\in N(u)$.
If $w_2v\notin E(G)$, then by Observation \ref{1} we have $|N(w_2)\cap N(v)|\ge 3$, and thus there exists a vertex $u_3\in N(w_2)\cap N(v)\setminus\{u,w_1\}$.
Since $u_1\neq u_2$, we may assume $u_1\neq u_3$.
Then $u w_1 u_1 w_2 u_3 v$ is a path of length five with $d(u) = d(v)$, a contradiction.
This proves $w_2v\in E(G)$.
If $uv\in E(G)$, then since $w_3\in A_u\cup B_{uv}\cup A_v$ and $u_1\neq u_2$, we may assume without loss of generality that $w_3\in N(u)$ and $w_3\neq u_1$.
Then $w_3 u v w_2 u_1 w_1$ is a path of length five with $d(w_3) = d(w_1)=n$, a contradiction.
This proves $uv\notin E(G)$.
If $w_1w_2\notin E(G)$, then by Observation \ref{1} there exists a vertex $u_3\in N(w_2)\cap N(v)\setminus\{w_1,u\}$.
Since $u_1\neq u_2$, we may assume $u_3\neq u_1$.
Then $u w_1 u_1 w_2 u_3 v$ is a path of length five with $d(u) = d(v)$, a contradiction.
This proves $w_1w_2\in E(G)$.

Since $d(u)=d(v)\ge n+1$ and $uv\notin E(G)$, by Observation \ref{1} we obtain $|N(u)\cap N(v)|\ge 4$.
Thus there exists a vertex $u_3\in (N(u)\cap N(v))\setminus\{w_1,w_2,w_3\}$.
Note that $w_3\in A_u\cup B_{uv}\cup A_v$; without loss of generality assume $w_3\in N(u)$.
Then $w_3 u u_3 v w_2 w_1$ is a path of length five with $d(w_3) = d(w_1)$, a contradiction.
\end{proof}

Let $w_1$ and $w_2$ be the two vertices with the maximum degrees in $A_u \cup B_{uv} \cup A_v$.
By Claims~\ref{e2n+4} and~\ref{e2.15}, it follows that for any vertex $z \in (A_u \cup B_{uv} \cup A_v) \setminus \{w_1, w_2\}$, $d(z) \le n$, and $d(w_1) + d(w_2) \le 2n + 3$.
Let $w_3$ and $w_4$ be the two vertices with the maximum degrees in $(A_u \cup B_{uv} \cup A_v) \setminus \{w_1, w_2\}$; then $d(w_3) + d(w_4) \le 2n$.
Furthermore, by Claim~\ref{enn}, we obtain that for any vertex $w \in (A_u \cup B_{uv} \cup A_v) \setminus \{w_1, w_2, w_3, w_4\}$, $d(w) \le n - 1$.
Thus, we have
\begin{align}
    \sum_{v \in A_u \cup B_{uv} \cup A_v} d(v) &\le (2n - 2 - |D_{uv}| - 4)(n - 1) + d(w_1) + d(w_2) + d(w_3) + d(w_4) \notag\\&
    \le (2n - 2 - |D_{uv}| - 4)(n - 1)+4n+3 \notag\\
    &\le 2n^2 + (2c - x - 2 - 2\cdot\mathbf{1}_{uv})n + x - 2c + 7 + 2\cdot\mathbf{1}_{uv}.\label{aba}
\end{align}
By Claim \ref{e2.10}, every vertex in $D_{uv}$ has degree at most $n-1$  if $G$ is not $K_{n-1,n+1}$. Thus,
\begin{align}
    \sum_{v\in D_{uv}} d(v) \le |D_{uv}| \,( n-1) .\label{d}
\end{align}
Now, using \eqref{aba}, \eqref{d}, and $d(u) = d(v) = n + c$, we can derive that
\begin{align*}
\sum_{z \in V(G)} d(z) &= \sum_{z \in A_u \cup B_{uv} \cup A_v} d(z) + \sum_{z \in D_{uv}} d(z) + d(u) + d(v) \\
&\le 2n^2 - 2n + 2c + 9.
\end{align*}
Since $2e(G) \ge 2(n^2 -1)$, we deduce
\(2c+11-2n\ge 0.\)
By our assumption that $n \ge 13$, and since $c$ is an integer, it follows that $c \ge \frac{2n-11}{2} > 7$.
By Observation \ref{1}, we have $x = |B_{uv}| \ge 16$.
Then by Lemma~\ref{dudv}, for any $w \in B_{uv}$, we have $d(w) \le 2n - (n + c) + 3 = n - c + 3 < n-1$.
Thus, we obtain
\begin{align}
\sum_{v \in A_u \cup B_{uv} \cup A_v} d(v)
&\le |B_{uv}|(n - c + 3) + (2n - 6 - |B_{uv}| - |D_{uv}|)(n - 1) + d(w_1) + d(w_2) + d(w_3) + d(w_4) \notag\\
&\le 2n^2 + (2c - 2 - 2\cdot\mathbf{1}_{uv} - x)n + 2\cdot\mathbf{1}_{uv} + (5 - c)x - 2c + 7.\label{2aba}
\end{align}
Now, using \eqref{d}, \eqref{2aba}, and $d(u) = d(v) = n + c$, we can derive that
\begin{align*}
\sum_{v \in V(G)} d(v) &= \sum_{v \in A_u \cup B_{uv} \cup A_v} d(v) + \sum_{v \in D_{uv}} d(v) + d(u) + d(v) \\
&\le 2n^2 - 2n + (4 - c)x + 2c + 9 \\
&\le 2n^2 - 2n + 2c - 10.
\end{align*}
The last inequality holds as $c > 7$ and $x \ge 16$.
Since $\sum_{v\in V(G)} d(v) = 2e(G) \ge 2(n^2 - 1)$, we have
\(
2c - 2n - 8 \ge 0.
\)
This implies $c > n$, but note that $d(u) = d(v) = n + c$, which is impossible.
This completes the proof of Lemma~\ref{ebeta 1/2}.
\end{proof}

\subsection{Proof of Lemma \ref{ebeta n+2}}
In this subsection, we prove Lemma \ref{ebeta n+2}. We first give the following easy bounds for $\beta$.
\begin{lem}\label{ebeta=4}
We have $5\leq \beta\leq \Delta$.
\end{lem}
\begin{proof}[{\bf Proof}]
The upper bound $\beta\le\Delta$ is immediate from the definitions. For the lower bound, assume $\beta\le4$. Then every degree greater than $4$ appears at most once, so the degree sum satisfies
\[
2e(G)=\sum d(v)\le 4\cdot5+\sum_{k=5}^{2n-1}k = 20+\frac{(4+2n)(2n-5)}{2}=2n^2-n+6<2n^2-2,
\]
where the last inequality holds for $n \ge 13$. This leads to a contradiction as $e(G)\ge n^2 -1$ edges.
\end{proof}

\begin{proof}[{\bf Proof Lemma \ref{ebeta n+2}}]
    Assume, for the sake of contradiction, that $\beta \le \Delta-4$ and $\Delta \ge n+4$.
Then there is a unique vertex $v_0$ of degree $\Delta$. Partition its neighborhood into two parts:
\[
A := \{v\in N(v_0) \mid d(v) \ge 2n-\Delta+4\},\qquad 
B := \{v\in N(v_0) \mid d(v) \le 2n-\Delta+3\}.
\]
If two vertices of $A$ had the same degree, then by Lemma~\ref{dudv} we would have
\[
\Delta = d(v_0) \le 2n - (2n-\Delta+4) + 3 = \Delta-1,
\]
a contradiction. Hence all vertices of $A$ have distinct degrees. The possible degrees for a vertex of $A$ lie in the interval $[2n-\Delta+4,\;\Delta-1]$, which contains exactly $2\Delta-2n-4$ integers. Consequently,
\[
|A| \le 2\Delta-2n-4.
\]
Since $|A| = \Delta - |B|$, we obtain
\begin{align}\label{e|B|ge}
    \Delta - |B| \le 2\Delta-2n-4 \quad\Longrightarrow\quad |B| \ge 2n-\Delta+4.
\end{align}

Let $\overline{N}(v_0) = V(G)\setminus(\{v_0\}\cup N(v_0))$ and note that $|\overline{N}(v_0)| = 2n-1-\Delta$. We aim to bound $\sum_{v\in V(G)} d(v)$ from above. The sum is split as
\begin{align}\label{esumv}
\sum_{v\in V(G)} d(v) = \Delta + \sum_{v\in A\cup\overline{N}(v_0)} d(v) + \sum_{v\in B} d(v).
\end{align}
The vertices of $B$ satisfy $d(v) \le 2n-\Delta+3$ by definition, so $\sum_{v\in B} d(v) \le |B|(2n-\Delta+3)$.

To bound $\sum_{v\in A\cup\overline{N}(v_0)} d(v)$, we maximise it under the following constraints:
\begin{itemize}
    \item[(a)] vertices of $A$ have pairwise distinct degrees;
    \item[(b)] no two vertices in $A\cup\overline{N}(v_0)$ share a degree greater than $\beta$.
\end{itemize}
Let $\lambda(2n,\Delta,\beta,|B|)$ denote the maximum possible sum of degrees for the $2n-1-|B|$ vertices in $A\cup\overline{N}(v_0)$. The degrees are chosen from $\{1,2,\dots,\Delta-1\}$. By Lemma \ref{lambda}, we obtain
\begin{align}\label{esumA}
\sum_{v\in V(G)} d(v) \le\lambda(2n, \Delta, \beta, |B|)
= \sum_{k=|B|}^{\Delta-1}k
+ \big(2n-1-\Delta\big)\beta
+ \binom{|B|-\beta}{2}
- \binom{(|B|-\beta)-(2n-1-\Delta)}{2}.
\end{align}

In what follows, we proceed our proof by distinguishing three cases according to the value of $|B|-\beta$.

\bigskip
{\bf Case 1}: $|B| - \beta\le 0$.
\medskip

By \eqref{esumv} and \eqref{esumA}, we obtain
\[
\sum_{v\in V(G)} d(v) \le \Delta + \sum_{k=|B|}^{\Delta-1} k + (2n-1-\Delta)\beta + |B|(2n-\Delta+3).
\]
The right‑hand side is a quadratic function in \(|B|\):
\[
f_1(|B|) = -\frac12|B|^2 + \Bigl(2n+\frac72-\Delta\Bigr)|B| + (2n-1-\Delta)\beta + \frac12\Delta^2 + \frac12\Delta.
\]
By \eqref{e|B|ge}, \(|B| \ge 2n-\Delta+4\). Since \(f_1\) is decreasing for \(|B| \ge 2n-\Delta+\frac72\), we have
\[
\sum_{v\in V(G)} d(v) \le f_1(2n-\Delta+4) = 2n^2+7n-2n\Delta+\Delta^2-3\Delta + (2n-1-\Delta)\beta +6.
\]
Using \(2e(G) \ge 2n^2-2\) yields
\begin{align}\label{ecase1}
    (2n-1-\Delta)\beta - 2n\Delta +7n + \Delta^2 -3\Delta +8 \ge 0.
\end{align}
If \(2n-1-\Delta =0\) then \eqref{ecase1} forces \(n\le12\), contradicting \(n\ge13\); thus \(2n-1-\Delta\ge1\) and
\[
\beta \ge  \frac{2n\Delta+3\Delta-7n-\Delta^2-8}{2n-1-\Delta} =\Delta-4 + \frac{n-12}{2n-1-\Delta} > \Delta-4,
\]
contradicting \(\beta\le\Delta-4\) for $n\ge 13$.

\bigskip
{\bf Case 2}:  $|B|-\beta\ge (2n-1-\Delta)+1$
\medskip

By \eqref{esumv} and \eqref{esumA}, we obtain
\begin{align*}
\sum_{v\in V(G)}  d(v) &\le \Delta + \sum_{k=|B|}^{\Delta-1}k
+ \big(2n-1-\Delta\big)\beta
+ \binom{|B|-\beta}{2}
- \binom{(|B|-\beta)-(2n-1-\Delta)}{2} + |B|(2n-\Delta+3)
\\& = -\frac12|B|^2 + \Bigl(4n-2\Delta+\frac52\Bigr)|B| -2n^2+2n\Delta + n.
\end{align*}
Treating this as a quadratic \(f_2(|B|)\), its maximum (for real \(|B|\)) occurs at \(|B| = 4n-2\Delta+\frac52\). Taking the nearest integer \(|B| = 4n-2\Delta+2\) gives an upper bound:
\begin{align}\label{ECASE2}
    \sum_{v\in V(G)} d(v) \le f_2(4n-2\Delta+2) = 2\Delta^2 - (6n+5)\Delta + 6n^2+11n+3.
\end{align}

From \(2e(G) \ge 2n^2-2\) we obtain
\[
2\Delta^2 - (6n+5)\Delta + 4n^2+11n+5 \ge 0.
\]
Solving this quadratic inequality in \(\Delta\) yields
\[
\Delta \le \frac{6n+5 - \sqrt{4n^2-28n-15}}{4} \quad\text{or}\quad \Delta \ge \frac{6n+5 + \sqrt{4n^2-28n-15}}{4}.
\]
The second lower bound exceeds \(2n-1\) (impossible), while the first upper bound is less than \(n+4\) for \(n\ge 9\), contradicting \(\Delta\ge n+4\).

\bigskip
{\bf Case 3}: $1\le |B|-\beta \le  (2n-1-\Delta)$.
\medskip

By \eqref{esumv} and \eqref{esumA}, we obtain
\begin{align}\label{ecase3}
    \sum_{v\in V(G)} d(v) &\le\Delta + \sum_{k=|B|}^{\Delta-1}k
+ \big(2n-1-\Delta\big)\beta
+ \binom{|B|-\beta}{2}
- \binom{(|B|-\beta)-(2n-1-\Delta)}{2} + |B|(2n-\Delta+3)\notag
\\& =\Delta + \sum_{k=\beta+1}^{\Delta-1} k + (2n-\Delta+\beta)\beta + |B|\bigl(2n-\Delta+3-\beta\bigr).
\end{align}

We distinguish two subcases according to the sign of \(2n-\Delta+3-\beta\).

\bigskip
{\bf Subcase 3a}: $2n-\Delta+3-\beta \ge 0$.
\medskip

The coefficient of \(|B|\) in \eqref{ecase3} is non‑negative, so the bound increases with \(|B|\). Using $|B|\le\Delta$ gives
\[
\sum_{v\in V(G)} d(v) \le \Delta + \sum_{k=\beta+1}^{\Delta-1} k + (2n-\Delta+\beta)\beta + \Delta(2n-\Delta+3-\beta).
\]
Simplifying (as detailed in the main text) yields a quadratic in \(\beta\):
\[
f_3(\beta) = \frac12\beta^2 + \Bigl(2n-2\Delta-\frac12\Bigr)\beta + 2n\Delta - \frac12\Delta^2 + \frac72\Delta.
\]
By Lemma~\ref{ebeta=4}, \(\beta\ge5\); the subcase condition implies $\beta\le2n-\Delta+3$. Since $f_3$ is convex, its maximum on $[5,\,2n-\Delta+3]$ occurs at an endpoint. Evaluating at the lower endpoint gives
\[
f_3(5) = -\frac12\Delta^2 + \Bigl(2n-\frac{13}2\Bigr)\Delta + 10n+10,
\]
and the inequality $2e(G)\ge2n^2-2$ would require 
\[
-\frac{1}{2}\Delta^2 + (2n-\frac{13}{2})\Delta - 2n^2 + 10n + 12 \ge 0,
\]
View the expression above as a quadratic polynomial in the variable \(\Delta\).
Then its discriminant
$\left(2n - \frac{13}{2}\right)^2 - 2(2n^2 - 10n - 12) = -6n + \frac{265}{4}$
must be non-negative, which is a contradiction for $n\ge12$. This is a contradiction. Hence \(f_3(5)\) cannot achieve the required lower bound, so the maximum must be at the other endpoint:
\[
f_3(2n-\Delta+3) = 2\Delta^2 - (6n+5)\Delta + 6n^2+11n+3,
\]
which is exactly \eqref{ECASE2}. The same analysis as in Case~2 shows that this expression is less than \(2n^2-2\) for \(\Delta \ge n+4\), leading to a contradiction.

\bigskip
{\bf Subcase 3b}: \(2n-\Delta+3-\beta < 0\).
\medskip

Now the coefficient of \(|B|\) in \eqref{ecase3} is negative, so the bound decreases with \(|B|\). Taking the smallest possible \(|B| = \beta+1\) yields
\[
\sum_{v\in V(G)} d(v) \le \Delta + \sum_{k=\beta+1}^{\Delta-1} k + (2n-\Delta+\beta)\beta + (\beta+1)(2n-\Delta+3-\beta).
\]
Simplifying gives another quadratic in \(\beta\):
\begin{align}\label{ecase3b}
    f_4(\beta) = -\frac12\beta^2 + \Bigl(4n-2\Delta+\frac32\Bigr)\beta + \frac12\Delta^2 - \frac12\Delta + 2n + 3.
\end{align}
This concave parabola attains its maximum at $\beta_0 = 4n-2\Delta+\frac52$. For integer $\beta$, the maximum value is achieved at $\beta = 4n-2\Delta+2$ or $4n-2\Delta+3$; substituting \(\beta = 4n-2\Delta+2\) gives
\[
f_4(4n-2\Delta+2) = \frac52\Delta^2 - \Bigl(8n+\frac72\Bigr)\Delta + 8n^2+8n+4.
\]
Using $2e(G)\ge2n^2-2$ we obtain
\[
\frac52\Delta^2 - \Bigl(8n+\frac72\Bigr)\Delta + 6n^2+8n+6 \ge 0.
\]
Solving this quadratic inequality in $\Delta$ yields either
\[
\Delta \ge \frac{16n+7 + \sqrt{16n^2-96n-191}}{10} > 2n-1
\quad\text{or}\quad
\Delta \le \frac{16n+7 - \sqrt{16n^2-96n-191}}{10} < \frac{4}{3}n + \frac{11}{6}.
\]
The first is impossible; the second gives \(\Delta < \frac{4}{3}n + \frac{11}{6}\). Together with \(\beta \le \Delta-4\) we have \(\beta \le \Delta-4 < 4n-2\Delta+\frac32\). Now substitute \(\beta = \Delta-4\) into \eqref{ecase3b}:
\[
f_4(\Delta-4) = -2\Delta^2 + (4n+13)\Delta -14n -11.
\]
The degree sum must satisfy
\[
-2\Delta^2 + (4n+13)\Delta -14n -11 \ge 2n^2-2,
\]
i.e.,
\[
-2\Delta^2 + (4n+13)\Delta -2n^2 -14n -9 \ge 0.
\]
The discriminant of the left‑hand side is $(4n+13)^2 - 8(2n^2+14n+9) = 97 - 8n$, which is negative for $n\ge13$. This is a contradiction.

%All cases lead to a contradiction, so our initial assumption was false. Hence either \(\beta \ge \Delta-3\) or \(\Delta \le n+3\).
Combining the above three cases, we conclude that our initial assumption is false. Hence, either \(\beta \ge \Delta-3\) or \(\Delta \le n+3\), completing the proof of Lemma \ref{ebeta n+2}.
\end{proof}

\end{document}